\documentclass[11pt]{amsart}
\usepackage{amsmath,amssymb,latexsym,esint,cite,mathrsfs}
\usepackage{verbatim,wasysym}
\usepackage[left=2.8cm,right=2.8cm,top=2.9cm,bottom=2.9cm]{geometry}

\usepackage{microtype}
\usepackage{color,enumitem,graphicx}
\usepackage[colorlinks=true,urlcolor=blue, citecolor=red,linkcolor=blue,
linktocpage,pdfpagelabels, bookmarksnumbered,bookmarksopen]{hyperref}
\usepackage[hyperpageref]{backref}
\usepackage[english]{babel}

\renewcommand{\H}{{\mathcal H}}
\newenvironment{properties}[1]{\begin{enumerate}

}{\end{enumerate}}

\newtheorem{lemma}{Lemma}[section]

\newtheorem{proposition}[lemma]{Proposition}
\newtheorem{theorem}[lemma]{Theorem}

\theoremstyle{definition}

\newtheorem{remark}[lemma]{Remark}

\newtheorem*{assumption}{(NA) Nonexistence Assumption}

\numberwithin{equation}{section}

\title[Global compactness for nonlocal problems]{Global compactness results for nonlocal problems}

\author[L.\ Brasco]{Lorenzo Brasco}
\author[M.\ Squassina]{Marco Squassina}
\author[Y.\ Yang]{Yang Yang}

\address[L.\ Brasco]{Dipartimento di Matematica e Informatica
\newline\indent
Universit\`a degli Studi di Ferrara
\newline\indent
Via Machiavelli 35, 44121 Ferrara, Italy}
\address{{\it and }
Aix-Marseille Universit\'e, CNRS
\newline\indent
Centrale Marseille, I2M, UMR 7373, 39 Rue Fr\'ed\'eric Joliot Curie
\newline\indent
13453 Marseille, France}
\email{lorenzo.brasco@unife.it}

\address[M.\ Squassina]{Dipartimento di Informatica
	\newline\indent
	Universit\`a degli Studi di Verona,
	Verona, Italy}
\email{marco.squassina@univr.it}

\address[Y. Yang]{School of Science
	\newline\indent
	Jiangnan University,
	Wuxi, Jiangsu 214122, China}
\email{yynjnu@126.com}

\begin{document}

\begin{abstract}
We obtain a Struwe type global compactness result for a class of nonlinear
nonlocal problems involving the fractional $p-$Laplacian operator and 
nonlinearities at critical growth. 
\end{abstract}

\maketitle

\begin{center}
	\begin{minipage}{8.5cm}
		\small
		\tableofcontents
	\end{minipage}
\end{center}

\medskip

\section{Introduction}
\subsection{Overview}
In the seminal paper \cite{struwe}, M.\ Struwe obtained a very useful
global compactness result for {\it Palais-Smale sequences} of the energy functional 
\[
I(u)=\frac{1}{2}\,\int_{\Omega} |\nabla u|^2\,dx+
\frac{\lambda}{2}\,\int_{\Omega} |u|^2\,dx-\frac{N-2}{2\,N}\,\int_{\Omega} |u|^\frac{2\,N}{N-2}\,dx, \qquad u\in D^{1,2}_0(\Omega)
\]
where $\Omega\subset\mathbb{R}^N$ is a smooth open and bounded set, $N\geq 3$,  $\lambda\in\mathbb{R}$, and the space $D^{1,2}_0(\Omega)$ is defined by
\[
D^{1,2}_0(\Omega)=\left\{u\in L^\frac{2\,N}{N-2}(\mathbb{R}^N)\, :\, \int_{\mathbb{R}^N} |\nabla u|^2\,dx<+\infty,\ u=0 \mbox{ in }\mathbb{R}^N\setminus \Omega\right\}.
\]
The functional above is naturally associated with the semi-linear elliptic problem with critical nonlinearity
\begin{equation}
\label{prob-ss}
\begin{cases}
-\Delta u+\lambda\,u=|u|^{\frac{4}{N-2}}\,u  & \text{in $\Omega$}, \\
u=0 & \text{on $\partial\Omega$},
\end{cases}
\end{equation}
in the sense that critical points of $I$ are weak solutions of \eqref{prob-ss}. 
Due to the presence of the term with critical growth in its definition, the functional $I$ does not satisfy the {\it Palais-Smale condition}. In other words, sequences $\{u_n\}_{n\in\mathbb{N}}\subset D^{1,2}_0(\Omega)$ of ``almost'' critical points of $I$ with bounded energy are not necessarily precompact in $D^{1,2}_0(\Omega)$.
Struwe's result gives a precise description of what happens when compactness fails at an energy level $c$.
Roughly speaking, in this case there exists a (possibly trivial) solution $v^0$
to \eqref{prob-ss} and $k$ profiles $v^k$ solving solving the purely critical problem on the whole space
\begin{equation}
\label{limiting}
-\Delta u=|u|^{\frac{4}{N-2}}u,\qquad \text{in } \mathbb{R}^N.
\end{equation}
such that the sequence $\{u_n\}_{n\in\mathbb{N}}$ can be ``almost'' written as a superposition of $v^0,\dots,v^k$.
More precisely, there exist $\{z_n^i\}_{n\in\mathbb{N}}\subset\mathbb{R}^N$ and $\{\lambda^i_n\}_{n\in\mathbb{N}}\subset\mathbb{R}_+$ converging to $0$ as $n\to\infty,$ with
$$
u_n \simeq v^0+\sum_{i=1}^k(\lambda_n^i)^{\frac{2-N}{2}}\,v^i\left(\frac{\cdot-z_n^i}{\lambda_n^i}\right),\qquad\text{ in }D^{1,2}_0(\mathbb{R}^N),
$$
and 
\begin{equation}
\label{equipartition}
c=I(v^0)+I_\infty(v^1)+\cdots+I_\infty(v^k), 
\end{equation}
where $I_\infty$ is the energy functional associated with equation \eqref{limiting}, i.e.
\[
I_\infty(u)=\frac{1}{2}\,\int_{\mathbb{R}^N} |\nabla u|^2\,dx-\frac{N-2}{2\,N}\,\int_{\mathbb{R}^N} |u|^\frac{2\,N}{N-2}\,dx.
\] 
This kind of result is very useful to study the existence of ground states
for nonlinear Schr\"odinger equations, Yamabe-type
equations or various classes of minimization problems. 
\par
Since then, several extensions of Struwe's result appeared in the literature for semi-linear elliptic 
problems. We refer the reader to \cite[Lemma 5]{ggs} for the case of the bilaplacian operator $\Delta^2$ with both Navier or Dirichlet boundary conditions 
and to \cite[Theorem 1.1]{palpis} for nonlocal problems involving the fractional Laplacian $(-\Delta)^s$ for $s\in (0,1)$. However, the linearity
of the operator does not seem essential in the derivation of this type of results. In fact in \cite[Theorem 1.2]{MW} (see also \cite{alves,yan}) 
a similar result was obtained for {\it signed} Palais-Smale sequences of the functional associated with the problem
\begin{equation*}
\begin{cases}
- \Delta_p u  +a\,|u|^{p-2}\,u=\mu\, |u|^{p^\ast - 2}\,u & \text{in $\Omega$}  \\
u  = 0  & \text{on $\partial\Omega$},
\end{cases}
\end{equation*}
where $a\in L^{N/p}(\Omega)$, $\mu>0$, $\Delta_p$ is the $p-$Laplacian operator and $p^*=N\,p/(N-p)$.
\par
Applications of these results are provided to constrained minimization problems, to {\it Br\'ezis--Nirenberg
type problems} (see \cite{MW}) and to {\it Bahri--Coron type problems} (see \cite{merc-sci-squ}), namely the existence of positive solutions
to the purely critical problem
\[
- \Delta_p u =\mu\, |u|^{p^\ast - 2}u,\qquad \mbox{ in }\Omega,
\]
when the domain $\Omega$ has a nontrivial topology. 
For the aforementioned results in the semi-linear case $p=2$, we also refer to the monograph \cite{willem-b}.

\subsection{Main results}
Let $1<p<\infty$ and $s\in(0,1)$.\ The aim of this paper is to obtain a global compactness result for Palais-Smale sequences of the $C^1$ nonlocal energy functional $I:D^{s,p}_0(\Omega)\to\mathbb{R}$ defined by
\begin{equation}
\label{I}
I(u):=\frac{1}{p}\,\int_{{\mathbb R}^{2N}} \frac{|u(x)-u(y)|^p}{|x-y|^{N+s\,p}}\,dx\,dy+\frac{1}{p}\,\int_{\mathbb{R}^N} a\,|u|^p\,dx-\frac{\mu}{p^*_s}\,\int_{\mathbb{R}^N}|u|^{p_s^*}\,dx,
\end{equation}
where $a\in L^{N/sp}(\Omega)$ and $\mu>0$ (see Section \ref{notations} below for the relevant definitions).
We recall that $\{u_n\}_{n\in\mathbb{N}}\subset D_0^{s,p}(\Omega)$ is said to be a {\em Palais-Smale sequence} for 
$I$ at level $c$ if
$$
\lim_{n\to\infty} I(u_n)=c, \qquad \lim_{n\to\infty} I'(u_n)=0\quad \text{ in } D^{-s,p'}(\Omega),
$$	
where $D^{-s,p'}(\Omega)$ denotes the topological dual space of $D^{s,p}_0(\Omega)$.
Critical points of \eqref{I} now solves (in weak sense)
\begin{equation}
\label{1}
\begin{cases}
(- \Delta)_p^s\, u  +a\,|u|^{p-2}u=\mu\,|u|^{p_s^\ast - 2}u & \text{in $\Omega$},  \\
u  = 0  & \text{in $\mathbb{R}^N \setminus \Omega$},
\end{cases}
\end{equation}
where $(-\Delta)_p^s$ is the {\it fractional $p-$Laplacian operator}, formally defined by 
\begin{equation*}
(- \Delta)^s_p\, u(x):= 2\, \lim_{\varepsilon \searrow 0} \int_{\mathbb{R}^N \setminus B_\varepsilon(x)}
\frac{|u(x) - u(y)|^{p-2}\, (u(x) - u(y))}{|x - y|^{N+s\,p}}\, dy, \qquad x \in \mathbb{R}^N.
\end{equation*}
 A function $u \in D^{s,p}_0(\Omega)$ is a weak solution of \eqref{1} if
\begin{align*}
\int_{\mathbb{R}^{2N}} \frac{|u(x) - u(y)|^{p-2}\, (u(x) - u(y))\, (v(x) - v(y))}{|x - y|^{N+s\,p}}\, dx\, dy&+\int_{\mathbb{R}^N} a\,|u|^{p-2}\,u\,v\, dx \\
&= \mu\int_{\mathbb{R}^N} |u|^{p_s^\ast - 2}\,u\,v\, dx, \quad \forall v \in D^{s,p}_0(\Omega).
\end{align*}
Likewise, if $\mathcal{H}$ is the whole $\mathbb{R}^N$ or is a half-space in $\mathbb{R}^N$, the critical points $u$ of the functional $I_\infty:D^{s,p}_0(\mathcal{H})\to\mathbb{R}$ defined by
\begin{equation}
\label{Infty}
I_\infty(u):=\frac{1}{p}\,\int_{{\mathbb R}^{2N}} \frac{|u(x)-u(y)|^p}{|x-y|^{N+s\,p}}\,dx\,dy-
\frac{\mu}{p^*_s}\int_{\mathbb{R}^N}|u|^{p_s^*}dx,
\end{equation}
are weak solutions to
\begin{equation}
\label{crittico}
\begin{cases}
(- \Delta)_p^s\, u  =  \mu\,|u|^{p_s^\ast - 2}u & \text{in $\H$,}  \\
u=0 & \text{in $\mathbb{R}^N\setminus \H$}.
\end{cases}
\end{equation}
\begin{assumption}
If $\mathcal{H}$ is a half-space, then \eqref{crittico} has the trivial solution only.
\end{assumption}
Our main result is the following
\begin{theorem}
\label{Theorem3}
We assume hypothesis {\bf (NA)}. Let $1<p<\infty$ and $s\in(0,1)$ be such that $s\,p<N$. Let $\Omega\subset\mathbb{R}^N$ be an open bounded set with smooth boundary. Let $\{u_n\}_{n\in\mathbb{N}}\subset D_0^{s,p}(\Omega)$ be a Palais-Smale sequence at level $c$ for the functional $I$ defined in \eqref{I}.
\par
Then there exist: 
\begin{itemize}
\item a (possibly trivial) solution $v^0\in D_0^{s,p}(\Omega)$ of
\[
(- \Delta)_p^s\, u +a\,|u|^{p-2}\,u = \mu\,|u|^{p_s^\ast - 2}\,u,\qquad \mbox{ in } \Omega;
\]
\item a number $k\in\mathbb{N}$ and $v^1,v^2\cdots,v^k\in D^{s,p}_0(\mathbb{R}^N)\setminus\{0\}$ solutions of
\[
(- \Delta)_p^s\, u  =  \mu\,|u|^{p_s^\ast - 2}u,\qquad \text{ in } \mathbb{R}^N;
\] 
\item a sequence of positive real numbers $\{\lambda_n^i\}_{n\in\mathbb{N}}\subset \mathbb{R}_+$ with $\lambda_n^i\to 0$ and a sequence of points $\{z_n^i\}_{n\in\mathbb{N}}\subset \{x\in\Omega\, :\, \mathrm{dist}(x,\partial\Omega)\ge \lambda^i_n\}$, for $i=1,\dots,k$; 
\end{itemize}
\vskip.2cm
such that, up to a subsequence, 
\begin{equation} 
\label{decomposizione}
\lim_{n\to\infty} \left[u_n-v^0-\sum_{i=1}^k(\lambda_n^i)^{\frac{sp-N}{p}}v^i
\Big(\frac{\cdot-z_n^i}{\lambda_n^i}\Big)\right]_{D^{s,p}(\mathbb{R}^N)}\!\!=0,
\end{equation} 
\begin{equation}
\label{decomposizione2}
\lim_{n\to\infty}[u_n]^p_{D^{s,p}(\mathbb{R}^N)}=\sum_{i=0}^k [v^i]^p_{D^{s,p}(\mathbb{R}^N)},
\end{equation}
\begin{equation}
\label{decomposizione3}
I(v^0)+\sum_{i=1}^k I_\infty(v^i)=c.
\end{equation}
\end{theorem}
By recalling that for every $u\in D^{1,p}_0(\mathbb{R}^N)$ we have
\[
\lim_{s\nearrow 1 }(1-s)\, \int_{{\mathbb R}^{2N}} \frac{|u(x)-u(y)|^p}{|x-y|^{N+s\,p}}\,dx\,dy=C\, \int_{\mathbb{R}^N} |\nabla u|^p\,dx,
\]
for a constant $C=C(N,p)>0$, original Struwe's result formally corresponds to $p=2$ and $s=1$ in Theorem \ref{Theorem3}. 
\begin{remark}[About the Nonexistence Assumption]
The nonexistence of solutions to problem~\eqref{crittico} for half-spaces is already changelling in the local case, for $p\not=2$. Indeed, without sign hypothesis on the solution, this is still open for the $p-$Laplacian. The situation is made unclear due to absence of a suitable Poho\v zaev type identity for $p\neq 2$, as well as of a unique continuation result up to the boundary. On the contrary, if we assume solutions to have constant sign, in the local case then this has been proved in \cite[Theorem 1.1]{MW}. For $0<s<1$ and $p=2$,  the non-existence of {\em signed} continuous solutions was obtained in \cite[Corollary 1.6]{MMW}. 
\end{remark}

\noindent
Next we formulate the global compactness result for radially symmetric functions in a ball $B\subset\mathbb{R}^N$. 
Due to the geometric restrictions, the final outcome is more precise and free of Assumption {\bf (NA)}. 
\begin{theorem}[Radial case]
\label{radial-case}	
Let $N\ge 2$, $1<p<\infty$ and $s\in(0,1)$ be such that $s\,p<N$. Let $B\subset\mathbb{R}^N$ be a ball centered at the origin and assume that $a\in L^{N/sp}_{{\rm rad}}(B)$. Let $\{u_n\}_{n\in\mathbb{N}}\subset D_{0,{\rm rad}}^{s,p}(B)$  be a Palais-Smale sequence for $I$ at level $c$.
Then there exist:
\begin{itemize} 
\item a (possibly trivial) solution $v_0\in D_{0,{\rm rad}}^{s,p}(B)$ of 
\[
(- \Delta)_p^s\, u  +a\,|u|^{p-2}\,u=\mu\, |u|^{p_s^\ast - 2}\,u,\quad \mbox{ in }B,
\]
\item a number $k\in\mathbb{N}$ and $v^1,v^2\cdots,v^k\in D^{s,p}_{0,{\rm rad}}(\mathbb{R}^N)\setminus\{0\}$ solutions of
\[
(- \Delta)_p^s\, u  = \mu\, |u|^{p_s^\ast - 2}\,u,\quad \mbox{ in } \mathbb{R}^N,  
\]
\item a sequence $\{\lambda_n^i\}_{n\in\mathbb{N}}\subset \mathbb{R}_+$ with $\lambda_n^i\to 0$, for $i=1,\dots,k$;
\end{itemize} 
such that, up to a subsequence, we have
\begin{equation}
\label{azero-r}
\lim_{n\to\infty}\left[ u_n-v_0-\sum_{i=1}^k(\lambda_n^i)^{\frac{s\,p-N}{p}}v^i\left(\frac{\cdot}{\lambda_n^i}\right)\right]_{D^{s,p}(\mathbb{R}^N)}=0, 
\end{equation}
and conclusions \eqref{decomposizione2} and \eqref{decomposizione3}.
\end{theorem}
\begin{remark}[Radial case for $N=1$]
The previous results guarantees that, under the standing assumptions, a radial Palais-Smale sequence can concentrate only at the origin. This is due to the fact that functions in $D^{s,p}_{0,\mathrm{rad}}$ verify some extra compactness properties on annular regions $\mathcal{A}_{R_0,R_1}=\{x\, :\, R_0<|x|<R_1\}$, which go up to the exponent $p^*_s$ (and even beyond). More precisely, we have compactness of the embeddings
\[
D^{s,p}_{0,\mathrm{rad}}(B)\hookrightarrow L^{q}(\mathcal{A}_{R_0,R_1}),\qquad p^*_s\le q<p^\#_s,
\] 
where $p^\#_s$ is the critical Sobolev exponent in dimension $N=1$.
As for $N=1$ we have $p^*_s=p^\#_s$, compactness ceases to be true for $s\,p<N=1$ (see Proposition \ref{prop:radembedding} and Remark \ref{oss:boh!}). In the one-dimensional case, in Theorem \ref{radial-case} one would need ({\bf NA}) as above.
\end{remark}
We point out that, contrary to \cite{MW,willem-b}, on the weight function $a$ we merely assume it to be in $L^{N/sp}(\Omega)$, avoiding an additional
coercivity assumption (see \cite[condition (B), p.125]{willem-b})
which was used in \cite{MW,willem-b} to get the boundedness of the Palais-Smale sequence $\{u_n\}_{n\in\mathbb{N}}$.
\par
The proof by Struwe in \cite{struwe} is essentially based upon
iterated rescaling arguments, jointly with an extension procedure to show the non-triviality of the weak limits. The latter seems
hard to adapt to the nonlocal cases, namely when $s>0$ is not integer.
Thus we prove Theorem~\ref{Theorem3} by basically following the scheme of Clapp's paper \cite{clapp}. A delicate point will be proving that the weak limits appearing in the construction are non-trivial.
As a main ingredient, we use a Caccioppoli inequality for solutions of $(-\Delta)_p^s u=f$ (see Proposition~\ref{prop:caccioloc} below).
\begin{remark}[The case $p=2$]
In the Hilbertian setting, namely for $p=2$ and $0<s<N/2$, Theorem \ref{Theorem3} has been recently proved in \cite{palpis} by appealing to the so-called
{\em profile decomposition} of Gerard, see \cite{gerard}. The latter is a general result describing the compactness defects of general bounded sequences in $D^{s,2}_0(\mathbb{R}^N)$, which are not necessarily Palais-Smale sequences of some energy functional. See also \cite[Theorem 1.4]{palpis-0}, where
some improved fractional Sobolev embeddings are obtained.
We point out that for $p\not=2$ such an approach does not seem feasible. Indeed, the paper \cite{jaffard} suggests that the decomposition \eqref{decomposizione} should not be expected for a generic bounded sequence in $D^{s,p}_0(\mathbb{R}^N)$ (see \cite[page 387]{jaffard}).
We also observe that some form of the global compactness result of \cite{palpis}
was also derived in \cite{sss} in the study of Coron-type results in the fractional case.
\end{remark}

\begin{remark}
We also consider a version of the above theorem 
stated for Palais-Smale sequences with sign, namely Palais-Smale sequences $\{u_n\}_{n\in\mathbb{N}}$ 
with the additional property that the negative parts $\{(u_n)_-\}_{n\in\mathbb{N}}$ converges to zero in $L^{p^*_s}$.
This is particularly interesting if $c$ is a {\em minimax} type level (i.e.\ with mountain pass, saddle point or linking geometry). 
Indeed, in this case it is often possible to obtain
a Palais-Smale sequence with sign at level $c$ via deformation arguments
of Critical Point Theory, see \cite[Theorem 2.8]{willem-b}. 
\end{remark}

\vskip8pt
\subsection{Notations}
\label{notations}
For $1<p<\infty$ we consider the monotone function $J_p:\mathbb{R}^N\to\mathbb{R}^N$ defined by
\[
J_p(\xi):=|\xi|^{p-2}\,\xi,\qquad \xi\in\mathbb{R}^N.
\]
We recall that this satisfies
\begin{equation}
\label{lipschitz}
|J_p(\xi)-J_p(\eta)|\le \left\{\begin{array}{ll}
|\xi-\eta|^{p-1},&\mbox{ if } 1<p\le 2,\\
&\\
C_p\, (|\xi|+|\eta|)^{p-2}\,|\xi-\eta|,& \mbox{ if } p>2.
\end{array}
\right.
\end{equation}
We denote by $B_r(x_0)$ the $N-$dimensional open ball of radius $r$, centered at a point $x_0\in\mathbb{R}^N$.
The symbol $\|\cdot\|_{L^p(\Omega)}$ stands for the standard norm for the $L^p(\Omega)$ space. For a measurable function $u:\mathbb{R}^N\to\mathbb{R}$, we let
\[
[u]_{D^{s,p}(\mathbb{R}^N)} := \left(\int_{\mathbb{R}^{2N}} \frac{|u(x) - u(y)|^p}{|x - y|^{N+sp}}\, dx dy\right)^{1/p}
\]
be its Gagliardo seminorm. For $s\,p<N$, we consider the space
$$
D^{s,p}_0(\mathbb{R}^N) := \big\{u \in L^{p^*_s}(\mathbb{R}^N) : [u]_{D^{s,p}(\mathbb{R}^N)}<\infty\big\},\qquad \mbox{ where } p^*_s=\frac{N\,p}{N-s\,p},
$$
endowed with norm $[\,\cdot\,]_{D^{s,p}(\mathbb{R}^N)}$. If $\Omega\subset\mathbb{R}^N$ is an open set, not necessarily bounded, we consider
\[
D^{s,p}_0(\Omega) := \big\{u \in D^{s,p}_0(\mathbb{R}^N): \text{$u=0$ in $\mathbb{R}^N\setminus\Omega$}\big\},
\]
If $\Omega$ is bounded, then the imbedding $D^{s,p}_0(\Omega)\hookrightarrow L^r(\Omega)$ is continuous
for $1\le r \le p_s^\ast$ and compact for $1\le r< p_s^\ast$.
The space $D^{s,p}_0(\Omega)$ can be equivalently defined as the completion
of $C^\infty_0(\Omega)$ in the norm $[\,\cdot\,]_{D^{s,p}(\mathbb{R}^N)}$, provided $\partial\Omega$ is smooth enough.
Finally, we shall denote the localized Gagliardo seminorm by
$$
[u]_{D^{s,p}(\Omega)}:
=\left(\int_{\Omega\times\Omega} \frac{|u(x)-u(y)|^p}{|x-y|^{N+s\,p}}\,dx\,dy\right)^{1/p}.
$$
\vskip12pt
\noindent
{\bf Acknowledgments.}\
L.B. and M.S.\ are members of the Gruppo Nazionale per l'Analisi Matematica, la Probabilit\`a
e le loro Applicazioni (GNAMPA) of the Istituto Nazionale di Alta Matematica (INdAM).
Y.Y.\ was supported by NSFC (No.\ 11501252, 11571176), Tian Yuan Special Foundation (No.\ 11226116),
Natural Science Foundation of Jiangsu Province of China for Young Scholars (No.\ BK2012109).
Part of this manuscript was written during a visit of M.S.\ at the University of
Ferrara in November 2015 and a subsequent visit of L.B. at the University of Verona in
February 2016.\ The hosting institutions and their facilities are gratefully acknowledged.

\section{Preliminary results}

\subsection{Br\'ezis-Lieb type properties}
\noindent
We first recall the following result (see \cite[Theorem 1]{BL} and \cite[Lemma 3.2]{MW}).

\begin{lemma}
\label{proof1}
Let $1<q<\infty$ and let
$\{f_n\}_{n\in\mathbb{N}}\subset L^q(\mathbb{R}^k)$ be a bounded sequence, such that $f_n\to f$ almost everywhere. 
Then
$$
\lim_{n\to\infty}\Big(\|f_n\|_{L^q(\mathbb{R}^k)}^q-\|f_n-f\|_{L^q(\mathbb{R}^k)}^q\Big)=\|f\|_{L^q(\mathbb{R}^k)}^q.
$$
Furthermore,
\begin{equation}
\label{Jp}
\lim_{n\to\infty}\int_{\mathbb{R}^k} \big|J_q(f_n)-J_q(f_n-f)-J_q(f)\big|^{q'}\,dx=0.
\end{equation}
\end{lemma}



\noindent
The previous result implies the following splitting properties.

\begin{lemma}
\label{proof2}
Let $\{u_n\}_{n\in\mathbb{N}}\subset D^{s,p}_0(\mathbb{R}^N)$ be such that $u_n\rightharpoonup u$ in $D^{s,p}_0(\mathbb{R}^N)$ and $u_n\to u$ 
almost everywhere, as $n\to\infty$. Then:
\begin{properties}{i}
\item\label{i1} $[u_n]^p_{D^{s,p}(\mathbb{R}^N)}-[u_n-u]^p_{D^{s,p}(\mathbb{R}^N)}=[u]^p_{D^{s,p}(\mathbb{R}^N)}+o_n(1);$
\vskip.2cm
\item\label{i2} $J_{p^*_s}(u_n)-J_{p^*_s}(u_n-u)\to J_{p^*_s}(u)$,\quad \hbox{in $L^{(p^*_s)'}(\mathbb{R}^N)$};
\vskip.2cm
\item \label{i3} it holds
\begin{multline*}
\frac{J_p(u_n(x)-u_n(y))}{|x-y|^{\frac{N+sp}{p'}}}-\frac{J_p\big((u_n(x)-u(x))-(u_n(y)-u(y))\big)}{|x-y|^{\frac{N+sp}{p'}}}
\to \frac{J_p(u(x)-u(y))}{|x-y|^{\frac{N+sp}{p'}}} \quad \hbox{in $L^{p'}(\mathbb{R}^{2N})$}.
\end{multline*}
\end{properties}
\end{lemma}
\begin{proof}
Statement \ref{i1} follows by Lemma \ref{proof1} by choosing
\[
f_n=\frac{u_n(x)-u_n(y)}{|x-y|^{\frac{N+s\,p}{p}}},\qquad f=\frac{u(x)-u(y)}{|x-y|^{\frac{N+s\,p}{p}}},\qquad q=p,\qquad k=2\,N.
\] 
With the same choices, we can also obtain \ref{i3} from \eqref{Jp}.
Statement \ref{i2}  directly follows from \eqref{Jp} with the choices
\[
f_n=u_n,\qquad f=u,\qquad q=p^*_s,\qquad k=N,
\] 
once we recalled that a weakly convergent sequence in $D^{s,p}_0(\mathbb{R}^N)$ weakly converges in $L^{p^*_s}(\mathbb{R}^N)$ as well, thanks to Sobolev inequality.
This concludes the proof.
\end{proof}
\noindent
Let $I$ and $I_\infty$ be the functionals defined by \eqref{I} and \eqref{Infty}. We recall that $I\in C^1(D^{s,p}_0(\Omega))$, $I_\infty\in C^1(D^{s,p}_0(\H))$ and 
\begin{align*}
\langle I'(u),\varphi\rangle &=
\int_{\mathbb{R}^{2N}} \frac{J_p(u(x) - u(y))\, (\varphi(x) - \varphi(y))}{|x - y|^{N+s\,p}}\, dx\, dy \\
&+\int_\Omega a\,|u|^{p-2}\,u\,\varphi\, dx- \mu\int_\Omega |u|^{p_s^\ast -2}\,u\,\varphi\, dx,\qquad \forall \varphi \in D^{s,p}_0(\Omega), \\
\langle I_\infty'(u),\varphi\rangle &=
\int_{\mathbb{R}^{2N}} \frac{J_p(u(x) - u(y))\, (\varphi(x) - \varphi(y))}{|x - y|^{N+s\,p}}\, dx\, dy \\
&-\mu\int_{\mathbb{R}^N} |u|^{p_s^\ast -2}\,u\,\varphi\, dx, \qquad \forall \varphi\in D^{s,p}_0(\mathcal{H}).
\end{align*}

\noindent
In the following, we repeatedly use the inclusion
$D_0^{s,p}(\Omega)\hookrightarrow D^{s,p}_0(\mathbb{R}^N)$.


\begin{lemma}
\label{elem-b}
Let $a\in L^{N/sp}(\Omega)$, assume that $\{u_n\}_{n\in\mathbb{N}}$ is bounded in $L^{p^*_s}(\Omega)$ and that $u_n\to u$ almost everywhere in $\Omega$. Then
\[
\lim_{n\to\infty}\Big\|a\,\big(J_p(u_n)-J_p(u)\big)\Big\|_{L^{(p^*_s)'}(\Omega)}=0.
\]
\end{lemma}
\begin{proof}
Let us set 
\[
\psi:=|a|^{(p^*_s)'}\in L^{\sigma}(\Omega), \qquad \mbox{ with } \sigma=\frac{Np-N+s\,p}{s\,p^2}>1,
\]
and 
\[
\phi_n:=|J_p(u_n)-J_p(u)|^{(p^*_s)'}\subset  L^\frac{\sigma}{\sigma-1}(\Omega).
\]
It is not difficult to see that $\{\phi_n\}_{n\in\mathbb{N}}$ is bounded in $L^{\sigma/(\sigma-1)}(\Omega)$ and converges to $0$ almost everywhere in $\Omega$, thanks to the assumptions on $\{u_n\}_{n\in\mathbb{N}}$. Thus we obtain
\[
\begin{split}
\lim_{n\to\infty}\int_\Omega \left|a\,\big(J_p(u_n)-J_p(u)\big)\right|^{(p^*_s)'}\,dx&=\lim_{n\to\infty}\int_\Omega \psi\,\phi_n\,dx,
\end{split}
\]
and the last limit is zero. Indeed, by Young inequality and Fatou Lemma for every $0<\tau\ll 1$,
\[
\begin{split}
\frac{1}{\sigma\,\tau^{\sigma-1}}\,\int_\Omega\psi^\sigma\,dx&\le \liminf_{n\to\infty}\int_\Omega \left[\frac{1}{\sigma\,\tau^{\sigma-1}}\,\psi^\sigma+\frac{\sigma-1}{\sigma}\,\tau\,\phi_n^\frac{\sigma}{\sigma-1}-\psi\,\phi_n\right]\,dx\\
&\le  \frac{1}{\sigma\,\tau^{\sigma-1}}\,\int_\Omega\,\psi^\sigma\,dx+\frac{\sigma-1}{\sigma}\,\tau\,\left(\sup_{n\in\mathbb{N}}\int_\Omega \phi_n^\frac{\sigma}{\sigma-1}\,dx\right)-\limsup_{n\to\infty} \int_\Omega \psi\,\phi_n\,dx.
\end{split}
\]
This proves
\[
0\le \limsup_{n\to\infty} \int_\Omega \psi\,\phi_n\,dx\le \frac{\sigma-1}{\sigma}\,\tau\,\left(\sup_{n\in\mathbb{N}}\int_\Omega \phi_n^\frac{\sigma}{\sigma-1}\,dx\right),
\]
and by the arbitrariness of $\tau>0$, we get the conclusion.
\end{proof}

\noindent
Next we produce a Palais-Smale sequence for $I_\infty$ from a Palais-Smale sequence for $I$.
\begin{lemma}\label{proof4}
Let $\{u_n\}_{n\in\mathbb{N}}\subset D^{s,p}_0(\Omega)$ be a Palais-Smale sequence for $I$ at the level $c$. Assume that 
\begin{equation}
\label{daluno}
u_n\rightharpoonup u\ \mbox{ in }D_0^{s,p}(\Omega)\qquad \mbox{ and }\qquad u_n\to u\ \mbox{ a.\,e. in }\Omega.
\end{equation}
Then, passing if necessary to a subsequence, $\{v_n\}_{n\in\mathbb{N}}:=\{u_n-u\}_{n\in\mathbb{N}}\subset D^{s,p}_0(\Omega)$ is a Palais-Smale sequence for the functional $I_\infty$ at the level $c-I(u)$. Moreover, we have
\begin{equation}
\label{sottraispezza}
[v_n]^p_{D^{s,p}(\mathbb{R}^N)}=[u_n]^p_{D^{s,p}(\mathbb{R}^N)}-[u]^p_{D^{s,p}(\mathbb{R}^N)}+o_n(1).
\end{equation}
\end{lemma}
\begin{proof}
We first observe that \eqref{daluno} readily gives that $I'(u)=0$, i.e.\ $u$ is a critical point of $I$.
By definition and hypothesis \eqref{daluno}, we have that $\{|v_n|^p\}_{n\in{\mathbb N}}$ is bounded in $L^{p^*_s/p}(\Omega)$ and $v_n\to 0$ a.e.\ on $\Omega$. 
Thus it follows that
$|v_n|^p$ converges weakly in $L^{p^*_s/p}(\Omega)$ to $0$.
Since $a\in L^{(p^*_s/p)'}(\Omega)$, we can infer
\[
\lim_{n\to\infty}\int_\Omega a\,|v_n|^pdx=0.
\]
A similar argument, shows that
\[
\int_\Omega a\,|u_n|^pdx=\int_\Omega a\,|u|^pdx+o_n(1).
\]
By \ref{i1} of Lemma \ref{proof2} we also get
\[
[u_n]^p_{D^{s,p}(\mathbb{R}^N)}-[v_n]^p_{D^{s,p}(\mathbb{R}^N)}=[u]^p_{D^{s,p}(\mathbb{R}^N)}+o_n(1),
\] 
which is \eqref{sottraispezza}.
By using the three previous displays and Lemma~\ref{proof1} for $L^{p^*_s}(\Omega)$, we have 
\[
I_\infty(v_n)=I(v_n)+o_n(1)
=I(u_n)-I(u)+o_n(1)
=c-I(u)+o_n(1).
\]
Finally, by virtue Lemma \ref{elem-b} applied to the sequence $u_n-u$, we have
\[
\lim_{n\to\infty}\Big\|a\,J_p(u_n-u)\Big\|_{L^{(p^*_s)'}(\Omega)}=0,
\] 
and thus 
\[
I_\infty'(v_n)=I'(v_n)+\hat o_n(1),
\]
where $\hat o_n (1)$ denotes a sequence going to zero in $D^{-s,p'}(\Omega)$. By using assertions \ref{i2}, \ref{i3} and Lemma \ref{elem-b} we further get
\begin{align*}
I_\infty'(v_n)&=I'(v_n)+\hat o_n(1)
=I'(u_n)-I'(u)+\hat o_n(1)
=\hat o_n(1),
\end{align*}
and $\hat o_n (1)$ still denotes a sequence 
going to zero in $D^{-s,p'}(\Omega)$. This concludes the proof.
\end{proof}

\subsection{Scaling invariance and related facts}

The following result follows from a direct computation, we leave the verification to the reader.
\begin{lemma}[Scaling invariance]
\label{proof7}
For $z\in \Omega$ and $\lambda>0$, we set
\[
\Omega_{z,\lambda}:=\frac{\Omega-z}{\lambda}.
\]
Then, the following facts hold:
\begin{itemize}
\item\label{a1} if $u\in D^{s,p}_0(\Omega)$ and we set
\[
v_{z,\lambda}(x):=\lambda^{\frac{N-s\,p}{p}}\,u(\lambda\, x+z)\in D^{s,p}_0(\Omega_{z,\lambda}),
\]
then $[v_{z,\lambda}]_{D^{s,p}(\mathbb{R}^N)}=[u]_{D^{s,p}(\mathbb{R}^N)}$ and $\|v_{z,\lambda}\|_{L^{p^*_s}(\mathbb{R}^N)}=\|u\|_{L^{p^*_s}(\mathbb{R}^N)}$;
\vskip.2cm
\item\label{a2} if we set
$$
\widetilde w(x):=\lambda^{\frac{s\,p-N}{p}}\,w\left(\frac{x-z}{\lambda}\right),\qquad
\varphi_{z,\lambda}(x):=\lambda^{\frac{N-s\,p}{p}}\,\varphi(\lambda\, x+z),
$$
for $w,\varphi\in D^{s,p}_0(\mathbb{R}^N)$, then $\langle I'_\infty(\widetilde w),\varphi\rangle=\langle I'_\infty(w), \varphi_{z,\lambda}\rangle$ and
$$
\sup_{\varphi\in  D^{s,p}_0(\Omega)}\left|\left\langle I'_\infty(\widetilde w),\frac{\varphi}{[\varphi]_{D^{s,p}(\mathbb{R}^N)}}\right\rangle\right|=
\sup_{\varphi\in  D^{s,p}_0(\Omega_{z,\lambda})}\left|\left\langle I'_\infty(w),\frac{\varphi}{[\varphi]_{D^{s,p}(\mathbb{R}^N)}}\right\rangle\right|. 
$$
\end{itemize}
\end{lemma}

\noindent
Next, we transform a Palais-Smale sequence for $I_\infty$ into a new one via rescaling and localization.

\begin{lemma}[Scalings, case I]
\label{proof5}
Let $\{z_n\}_{n\in\mathbb{N}}\subset \Omega$ and $\{\lambda_n\}_{n\in\mathbb{N}}\subset \mathbb{R}_+$ be such that
\[
\lim_{n\to\infty} z_n=z_0\qquad \mbox{ and }\qquad \lim_{n\to\infty} \lambda_n=0.
\]
Assume that $\{u_n\}_{n\in\mathbb{N}}\subset D^{s,p}_0(\Omega)$ is a Palais-Smale sequence for $I_\infty$ at level $c$
and that the rescaled sequence 
\[
v_{n}(x):=\lambda_n^{\frac{N-s\,p}{p}}u_n(\lambda_n\, x+z_n)\in D^{s,p}_0(\Omega_n),\qquad\mbox{ where }\ \Omega_n:=\frac{\Omega-z_n}{\lambda_n},
\]
is such that
\[
v_n\rightharpoonup v \quad\text{in $D^{s,p}_0(\mathbb{R}^N)$},\qquad
v_n\to v \quad \hbox{a.e.\ in $\mathbb{R}^N$}.
\]
If
\begin{equation}
\label{sbattila}
\lim_{n\to\infty}\frac{\varrho_n}{\lambda_n}=+\infty,\qquad \mbox{ where } \ \varrho_n:=\frac{1}{2}\,\mathrm{dist}(z_n,\partial \Omega),
\end{equation}
then $v$ is a critical point of $I_\infty$ on $D^{s,p}_0(\mathbb{R}^N)$, i.e. 
\[
\langle I_\infty'(v),\varphi\rangle=0,\qquad \mbox{ for every }\varphi \in D^{s,p}_0(\mathbb{R}^N).
\] 
Moreover, if $\zeta\in C^\infty_0(B_2(0))$ is a standard cut-off such that $\zeta\equiv 1$ on $B_{1}(0)$, the sequence
\[
w_n(z):=u_n(z)-\lambda_n^{\frac{s\,p-N}{p}}v\left(\frac{z-z_n}{\lambda_n}\right)\,\zeta\left(\frac{z-z_n}{\varrho_n}\right)\,\in D^{s,p}_0(\Omega),
\]
is a Palais-Smale sequence for $I_\infty$ at level $c-I_\infty(v)$ and such that
\begin{equation}
\label{ia} 
[u_n]^p_{D^{s,p}(\mathbb{R}^N)}-[w_n]^p_{D^{s,p}(\mathbb{R}^N)}=[v]^p_{D^{s,p}(\mathbb{R}^N)}+o_n(1).
\end{equation}
\end{lemma}
\begin{proof}
Let us assume \eqref{sbattila}, under this assumption the sets $\Omega_n$ converges to $\mathbb{R}^N$. Thus, for every $\varphi\in C^\infty_0(\mathbb{R}^N)$ with compact support, we can assume that $\Omega_n$ contain the support of $\varphi$ for $n$ sufficiently large. 
From Lemma \ref{proof7} and the hypothesis on $\{u_n\}_{n\in\mathbb{N}}$, it readily follows 
\[
0=\lim_{n\to\infty}\left\langle I'_\infty(u_n),\lambda^{\frac{s\,p-N}{p}}_n\,\varphi\left(\frac{\cdot-z_n}{\lambda_n}\right)\right\rangle=\lim_{n\to\infty}\langle I'_\infty(v_n),\varphi\rangle=\langle I'_\infty(v),\varphi\rangle.
\]
By arbitrariness of $\varphi\in C^\infty_0(\mathbb{R}^N)$, we get the desired conclusion. Before going on, we observe that since $v$ is a critical point of $I_\infty$, from Lemma \ref{prop:integrabile} we get
\begin{equation}
\label{beccalossi}
v\in L^q(\mathbb{R}^N),\qquad \mbox{ for every } \frac{p^*_s}{p'}<q\le p^*_s.
\end{equation}
\vskip.2cm\noindent
For the second part of the statement, we first observe that $w_n\in D^{s,p}_0(\Omega)$ thanks to Lemma \ref{Bddlemma}.  Thanks to \eqref{beccalossi} we can apply Lemma \ref{lm:truncation}: by using this and \ref{i1} of Lemma \ref{proof2}, we have
\begin{equation}
\label{spezzamento}
\begin{split}
[v_n]^p_{D^{s,p}(\mathbb{R}^N)}&-[v_n-v\,\zeta(\lambda_n/\varrho_n\,\cdot)]^p_{D^{s,p}(\mathbb{R}^N)}\\
&=[v_n]^p_{D^{s,p}(\mathbb{R}^N)}-[v_n-v]^p_{D^{s,p}(\mathbb{R}^N)}+o_n(1)=[v]^p_{D^{s,p}(\mathbb{R}^N)}+o_n(1),
\end{split}
\end{equation}
thanks to the fact that $\lambda_n/\varrho_n$ converges to $0$, by assumption. 
From the scaling properties of Lemma \ref{proof7}, this yields
\[
[u_n]^p_{D^{s,p}(\mathbb{R}^N)}-[w_n]^p_{D^{s,p}(\mathbb{R}^N)}=[v]^p_{D^{s,p}(\mathbb{R}^N)}+o_n(1),\qquad\text{as $n\to\infty$,}
\]
which proves \eqref{ia}. Similarly to \eqref{spezzamento}, we also have
\begin{equation}
\label{spezzamento2}
\|v_n\|^{p^*_s}_{L^{p^*_s}(\mathbb{R}^N)}-\|v_n-v\,\zeta(\lambda_n/\varrho_n\,\cdot)\|^{p^*_s}_{L^{p^*_s}(\mathbb{R}^N)}=\|v\|^{p^*_s}_{L^{p^*_s}(\mathbb{R}^N)}+o_n(1)
\end{equation}
By scaling, \eqref{spezzamento} and \eqref{spezzamento2} we get
\begin{align*}
I_\infty(w_n)&=\frac{1}{p}\,[v_n-v\,\zeta(\lambda_n/\varrho_n\,\cdot)]^p_{D^{s,p}(\mathbb{R}^N)}-\frac{\mu}{p^*_s}\,\int_{\mathbb{R}^N} |v_n-v\,\zeta(\lambda_n/\varrho_n\,\cdot)|^{p^*_s}\,dx\\
&=\frac{1}{p}\,[v_n]_{D^{s,p}(\mathbb{R}^N)}-\frac{1}{p}\,[v]^p_{D^{s,p}(\mathbb{R}^N)}-\frac{\mu}{p^*_s}\,\int_{\mathbb{R}^N} |v_n|^{p^*_s}\,dx-\frac{\mu}{p^*_s}\,\int_{\mathbb{R}^N} |v|^{p^*_s}\,dx+o_n(1)\\
&=I_{\infty}(v_n)-I_\infty(v)+o_n(1)\\
&=I_\infty(u_n)-I_\infty(v)+o_n(1)\\
&=c-I_\infty(v)+o_n(1).
\end{align*}
It is only left to show that $\{w_n\}_{n\in\mathbb{N}}$ is a Palais-Smale sequence. For any $\varphi\in D^{s,p}_0(\Omega)$ with $[\varphi]_{D^{s,p}(\mathbb{R}^N)}=1$, we set
\[
\varphi_n(x)=\lambda_n^{\frac{N-s\,p}{p}}\,\varphi(\lambda_n\, x+z_n)\in D^{s,p}_0(\Omega_n).
\]
Clearly we still have $[\varphi_n]_{D^{s,p}(\mathbb{R}^N)}=1$.
We first observe that
\begin{equation}
\label{spezzalo}
\langle I'_\infty(v_n-v\,\zeta(\lambda_n/\varrho_n\,\cdot)),\varphi_n\rangle=\langle I_\infty'(v_n-v), \varphi_n\rangle+o_n(1),
\end{equation}
where $o_n(1)$ is independent of $\varphi$. Indeed, by using the compact notations
\[
Z_n(x,y)=\Big(v_n(x)-v(x)\,\zeta(\lambda_n/\varrho_n\,x)\Big)-\Big(v_n(y)-v(y)\,\zeta(\lambda_n/\varrho_n\,y)\Big),
\]
and
\[
V_n(x,y)=\Big(v_n(x)-v(x)\Big)-\Big(v_n(y)-v(y)\Big),
\]
we have
\[
\begin{split}
\Big|\langle I'_\infty(v_n-v\,\zeta(\lambda_n/\varrho_n\,\cdot))&-I_\infty'(v_n-v),\varphi_n\rangle\Big|\\
&\le \left|\int_{\mathbb{R}^{2N}} \frac{\Big(J_p(Z_n(x,y))-J_p(V_n(x,y))\Big)\,\Big(\varphi_n(x)-\varphi_n(y)\Big)}{|x-y|^{N+s\,p}}\,dx\,dy\right|\\
&+\mu\,\left|\int_{\mathbb{R}^N} \Big(J_{p^*_s}(v_n-v\,\zeta(\lambda_n/\varrho_n\cdot))-J_{p^*_s}(v_n-v)\Big)\,\varphi\,dx\right|.
\end{split}
\]
We focus on the nonlocal term, the other being easier. By H\"older inequality this is estimated by
\[
\left(\int_{\mathbb{R}^{2N}}\frac{|J_p(Z_n(x,y))-J_p(V_n(x,y))|^{p'}}{|x-y|^{N+s\,p}}\,dx\,dy\right)^\frac{1}{p'}. 
\]
Let us suppose for simplicity that\footnote{For $1<p\le 2$ the proof is even simpler, it is still sufficient to use \eqref{lipschitz}.} $p> 2$. Then we use \eqref{lipschitz} and H\"older inequality with exponents 
\[
\frac{p}{p'}\qquad \mbox{ and } \qquad \frac{p}{p-p'},
\]
so to get
\[
\begin{split}
\int_{\mathbb{R}^{2N}}&\frac{|J_p(Z_n(x,y))-J_p(V_n(x,y))|^{p'}}{|x-y|^{N+s\,p}}\,dx\,dy\\
&\le C_p\, \int_{\mathbb{R}^{2N}}\left(|Z_n(x,y)|+|V_n(x,y)|\right)^{p'\,(p-2)}\,\frac{|Z_n(x,y)-V_n(x,y)|^{p'}}{|x-y|^{N+s\,p}}\,dx\,dy\\
&\le C_p\, \left(\int_{\mathbb{R}^{2N}}\frac{\left(|Z_n(x,y)|+|V_n(x,y)|\right)^{p}}{|x-y|^{N+s\,p}}\,dx\,dy\right)^\frac{p-p'}{p}\,\left(\int_{\mathbb{R}^{2N}}\frac{|Z_n(x,y)-V_n(x,y)|^{p}}{|x-y|^{N+s\,p}}\,dx\,dy\right)^\frac{p'}{p}.
\end{split}
\]
By recalling the definitions of $Z_n$ and $V_n$, we get that the first term is uniformly bounded, while the second one coincides with
\[
[v\,\zeta(\lambda_n/\varrho_n\cdot)-v]^{p'}_{D^{s,p}(\mathbb{R}^N)},
\]
which converges to $0$ thanks to Lemma \ref{lm:truncation}. This proves \eqref{spezzalo} and by using it in conjunction with Lemma \ref{proof2}, we get
\begin{align*}
\langle I_\infty'(w_n), \varphi\rangle&=\langle I'_\infty(v_n-v\,\zeta(\lambda_n/\varrho_n\,\cdot)),\varphi_n\rangle\\
&=\langle I_\infty'(v_n-v), \varphi_n\rangle+o_n(1)\\
&=\langle I_\infty'(v_n),\varphi_n\rangle-\langle I_\infty'(v), \varphi_n\rangle+o_n(1),\\
&=\langle I_\infty'(u_n),\varphi\rangle-\langle I_\infty'(v), \varphi_n\rangle+o_n(1),
\end{align*}
where $o_n(1)$ is independent of $\varphi$. 
We now use that $\{u_n\}_{n\in\mathbb{N}}$ is a Palais-Smale sequence 
and that $\langle I_\infty'(v), \varphi_n\rangle=0$ by the first part of the proof. 
This allows us to conclude. 
\end{proof}

\begin{lemma}[Scalings, case II]
\label{proof6}
Under the assumptions of Lemma \ref{proof5}, if 
\begin{equation}
\label{risbattila}
\liminf_{n\to\infty}\frac{1}{\lambda_n}\,\mathrm{dist}(z_n,\partial \Omega)<\infty.
\end{equation}
then $z_0\in\partial\Omega$, $v\in D^{s,p}_0(\mathcal{H})$ and $v$ is a critical point of $I_\infty$ on $D^{s,p}_0(\mathcal{H})$, i.e.
\[
\langle I_\infty'(v),\varphi\rangle=0,\qquad \mbox{ for every } \varphi\in D^{s,p}_0(\mathcal{H}),
\] 
where $\mathcal{H}$ is a half-space.
\end{lemma}
\begin{proof}
Under the assumption \eqref{risbattila}, the proof is the same as in the first part of Lemma \ref{proof5}, we only have to observe that in this case the sets $\Omega_n$ converge to a half-space $\mathcal{H}$.
\end{proof}

\noindent
Next we prove that nonsingular scalings of weakly vanishing sequences are weakly vanishing. 
\begin{lemma}
\label{weak-c}
Assume that $u_n\rightharpoonup 0$ in $D^{s,p}_0(\mathbb{R}^N)$, $\lambda_n\to \lambda_0>0$,
$\{z_n\}_{n\in\mathbb{N}}\subset\mathbb{R}^N$ such that $z_n\to z_0$. We set
\[
v_n(x):=\lambda_n^{\frac{N-sp}{p}}u_n(\lambda_nx+z_n).
\]
Then $v_n\rightharpoonup 0$ in $D^{s,p}_0(\mathbb{R}^N)$.
\end{lemma}
\begin{proof}
Take any continuous functional $F\in D^{-s,p'}(\mathbb{R}^N)$. Then, there exists a function $\varphi\in L^{p'}(\mathbb{R}^{2N})$ with
$$
\langle F,u\rangle= \int_{\mathbb{R}^{2N}}\frac{\varphi(x,y)\,(u(x)-u(y))}{|x-y|^{\frac{N+s\,p}{p}}}\,dx\,dy,\qquad \text{for all }u\in D^{s,p}_0(\mathbb{R}^N).
$$
We have, by a change of variables,
\begin{align*}
 \langle F,v_n\rangle &=\lambda_n^{\frac{N-s\,p}{p}}\int_{\mathbb{R}^{2N}}\frac{\varphi(x,y)(u_n(\lambda_n\,x+z_n)-u_n(\lambda_n\,y+z_n))}
{|x-y|^\frac{N+s\,p}{p}}\,dx\,dy\\
&=\lambda_n^{-\frac{2N}{p'}}\int_{\mathbb{R}^{2N}}\frac{\varphi\left(\dfrac{x-z_n}{\lambda_n},\dfrac{y-z_n}{\lambda_n}\right)(u_n(x)-u_n(y))}{ |x-y|^{\frac{N+sp}{p}}}\,dx\,dy :=\omega_n.
\end{align*}
On the other hand, introducing the functions of $\Psi_n,\Psi\in L^{p'}(\mathbb{R}^{2N})$ by setting
$$
\Psi_n(x,y):=\varphi\left(\frac{x-z_n}{\lambda_n},\frac{y-z_n}{\lambda_n}\right),\qquad
\Psi(x,y):=\varphi\left(\frac{x-z_0}{\lambda_0},\frac{y-z_0}{\lambda_0}\right),
$$
we have
\begin{align*}
\omega_n&=\lambda_n^{-\frac{2N}{p'}}\int_{\mathbb{R}^{2N}}\frac{\Psi(x,y)(u_n(x)-u_n(y))}{|x-y|^\frac{N+s\,p}{p}}\,dxdy   \\
&+\lambda_n^{-\frac{2N}{p'}}\int_{\mathbb{R}^{2N}}\frac{(\Psi_n(x,y)-\Psi(x,y))(u_n(x)-u_n(y))}{|x-y|^\frac{N+s\,p}{p}}\,dxdy   \\
& =\lambda_n^{-\frac{2N}{p'}}\int_{\mathbb{R}^{2N}}\frac{(\Psi_n(x,y)-\Psi(x,y))(u_n(x)-u_n(y))}{|x-y|^\frac{N+s\,p}{p}}\,dxdy+o_n(1),
\end{align*}
in view of $u_n\rightharpoonup 0$ in $D^{s,p}_0(\mathbb{R}^N)$ and $\Psi\in L^{p'}(\mathbb{R}^{2N})$. Then $\omega_n=o_n(1)$ follows by 
$$
\sup_{n\in\mathbb{N}}\left\|\frac{u_n(x)-u_n(y)}{|x-y|^{\frac{N+sp}{p}}}\right\|_{L^{p}(\mathbb{R}^{2N})}<\infty,
$$
and $\Psi_n\to\Psi$ strongly in $L^{p'}(\mathbb{R}^{2N})$ as $n\to\infty$, since
$\lambda_n\to \lambda_0>0$ and $z_n\to z_0$.
\end{proof}

\subsection{Estimates for solutions}

\noindent
Next we prove a Caccioppoli inequality, which will turn out to be the main technical
tool in order to handle {\bf Step 3} in the proof of Theorem~\ref{Theorem3}.

\begin{proposition}[Caccioppoli inequality]
\label{prop:caccioloc}
Let $F\in D^{-s,p'}(\Omega)$ and let $u\in D^{s,p}_0(\Omega)$ with
\begin{equation*}
\int_{\mathbb{R}^{2\,N}}\frac{J_p(u(x)-u(y))\,(\varphi(x)-\varphi(y))}{|x-y|^{N+s\,p}}\,dx\,dy
=\langle F,\varphi\rangle,\quad \text{for any $\varphi\in D^{s,p}_0(\Omega)$.}
\end{equation*}
Then for every open set $\Omega'$ such that $\Omega'\cap \Omega\not=\emptyset$ and every positive $\psi\in C^\infty_0(\Omega')$ we have
\begin{equation*}
\begin{split}
\int_{\Omega'\times\Omega'} &\frac{\big|u(x)\,\psi(x)-u(y)\,\psi(y)\big|^{p}\,}{|x-y|^{N+s\,p}}\, dx\, dy\\
&\le \mathcal{C}\,\int_{\Omega'\times\Omega'} \frac{|\psi(x)-\psi(y)|^p}{|x-y|^{N+s\,p}}\,\Big(|u(x)|^p+|u(y)|^p\Big)\, dx\,dy\\
&+\mathcal{C}\,\left(\sup_{y\in \mathrm{spt}(\psi)} \int_{\mathbb{R}^N\setminus \Omega'} \frac{|u(x)|^{p-1}}{|x-y|^{N+s\,p}}\,dx\right)\, \int_{\Omega'} |u|\, \psi^p\, dx+\mathcal{C}\,\Big|\langle F,u\,\psi^p\rangle\Big|,
\end{split}
\end{equation*}
for some constant $\mathcal{C}>0$ depending on $p$ only.
\end{proposition}
\begin{proof}
The proof is the same as that of Caccioppoli inequality \cite[Proposition 3.5]{BP}. The only differences are that here $F$ is not necessarily (represented by) a function and that the test function $\psi$ can cross the boundary $\partial\Omega$.
We insert the test function\footnote{Observe that this is a legitimate test function, since $\psi^p\,u\in D^{s,p}_0(\mathbb{R}^N)$ by Lemma \ref{Bddlemma} and $\psi^p\,u \equiv 0$ outside $\Omega$.} 
$\varphi=\psi^p\,u$, where $\psi\in C^\infty_0(\Omega)$ is as in the statement. Then we get
\begin{equation}
\label{subeigeneq20}
\begin{split}
\int_{\mathbb{R}^{2N}} &\frac{J_p(u(x)-u(y))}{|x-y|^{N+s\,p}}\, (u(x)\,\psi(x)^p-u(y)\,\psi(y)^p)\, dx\, dy=\langle F,u\,\psi^p\rangle.
\end{split}
\end{equation}
We now split the double integral in three parts:
\[
\mathcal{I}_1=\int_{\Omega'\times\Omega'} \frac{J_p(u(x)-u(y))}{|x-y|^{N+s\,p}}\, (u(x)\,\psi(x)^p-u(y)\,\psi(y)^p)\, dx\, dy,
\]
\[
\mathcal{I}_2=\int_{\Omega'\times(\mathbb{R}^N\setminus\Omega')} \frac{J_p(u(x)-u(y))}{|x-y|^{N+s\,p}}\,u(x)\,\psi(x)^p\, dx\, dy,
\]
and
\[
\mathcal{I}_3=-\int_{\Omega'\times(\mathbb{R}^N\setminus\Omega')} \frac{J_p(u(x)-u(y))}{|x-y|^{N+s\,p}}\, u(y)\,\psi(y)^p\, dx\, dy
\]
The first integral $\mathcal{I}_1$ can be estimated exactly as in \cite[Proposition 3.5]{BP}, with the choices
\[
v=u,\qquad g(t)=t=G(t),
\]
there. This gives
\begin{equation}
\label{I1}
\begin{split}
c\,\int_{\Omega'\times\Omega'} &\frac{\Big|u(x)\,\psi(x)-u(y)\,\psi(y)\Big|^{p}\,}{|x-y|^{N+s\,p}}\,dx\, dy\\
&\le \mathcal{I}_1+C\,\int_{\Omega'\times\Omega'} \frac{|\psi(x)-\psi(y)|^p}{|x-y|^{N+s\,p}}\,\Big(|u(x)|^p+|u(y)|^p\Big)\, dx\,dy.\\
\end{split}
\end{equation}
For the estimate of $\mathcal{I}_2$ we proceed similarly to \cite{BP}, by observing that the positivity assumption on $u$ can be dropped. Namely, we simply observe that by monotonicity of $\tau\mapsto J_p(\tau)$, for $x\in\Omega'$ we have
\[
J_p(u(x)-u(y))\ge J_p(-u(y)),\qquad \mbox{ if } u(x)\ge 0
\]
or
\[
J_p(u(x)-u(y))\le J_p(-u(y)),\qquad \mbox{ if } u(x)<0.
\]
Thus in both cases we get
\[
J_p(u(x)-u(y))\,u(x)\ge J_p(-u(y))\,u(x).
\]
Then we obtain
\begin{equation}
\label{I2}
\begin{split}
\mathcal{I}_2&\ge -\int_{\Omega'\times(\mathbb{R}^N\setminus\Omega')} \frac{|u(y)|^{p-2}\, u(y)}{|x-y|^{N+s\,p}}\, u(x)\,\psi(x)^p\, dx\, dy\\
&\ge-\left(\sup_{x\in\mathrm{spt}(\psi)} \int_{\mathbb{R}^N\setminus \Omega'} \frac{|u(y)|^{p-1}}{|x-y|^{N+s\,p}}\,dy\right)\, \int_{\Omega'} |u(x)|\, \psi(x)^p\, dx.
\end{split}
\end{equation}
The third integral can be estimated in a similar fashion. By inserting the above estimates in \eqref{subeigeneq20}, we get the conclusion.
\end{proof}
\noindent
Let us set
\[
\mathcal{S}_{p,s}:=\inf_{u\in D^{s,p}_0(\mathbb{R}^N)} \left\{ [u]_{D^{s,p}(\mathbb{R}^N)}^p\, :\, \|u\|_{L^{p^*_s}(\mathbb{R}^N)}=1\right\},
\]
which is nothing but the sharp constant in the Sobolev inequality for $D^{s,p}_0(\mathbb{R}^N)$, namely
\begin{equation}
\label{sobolevity}
\mathcal{S}_{p,s}\,\mathcal\|u\|_{L^{p^*_s}(\mathbb{R}^N)}^p\le [u]^p_{D^{s,p}(\mathbb{R}^N)},\qquad \text{for all $u\in D^{s,p}_0(\mathbb{R}^N)$}.
\end{equation}
It is useful to remark that if $u\in D^{s,p}_0(E)$ weakly solves
\begin{equation}
\label{2.10}
\begin{cases}
(- \Delta)_p^s\, u  = \mu\, |u|^{p_s^\ast - 2}u & \text{in }E  \\
u=0 & \text{in $\mathbb{R}^N\setminus E$},
\end{cases}
\end{equation}
in some open set $E\subset\mathbb{R}^N$ ($E=\mathbb{R}^N$ is allowed) and for some $\mu>0$,
then we get
\[
[u]_{D^{s,p}(\mathbb{R}^N)}^p=\mu\,\|u\|_{L^{p^*_s}(E)}^{p^*}.
\]
Combining this with \eqref{sobolevity} yields the following universal lower bounds for the norms 
of the nontrivial solutions of problem \eqref{2.10}, that is 
\begin{equation}
\label{snorme}
\|u\|_{L^{p^*_s}(E)}^{p^*}\ge \left(\frac{\mathcal{S}_{p,s}}{\mu}\right)^\frac{N}{s\,p}\qquad \mbox{ and }\qquad [u]^p_{D^{s,p}(\mathbb{R}^N)}\ge \mu\,\left(\frac{\mathcal{S}_{p,s}}{\mu}\right)^\frac{N}{s\,p}.
\end{equation}
This in turn entails the following universal estimate for the energy of solutions
\[
\frac{1}{p}\,[u]^p_{D^{s,p}(\mathbb{R}^N)}-\frac{\mu}{p^*_s}\,\int_{E} |u|^{p^*_s}\,dx\ge  \mu\,\frac{s}{N}\,\left(\frac{\mathcal{S}_{p,s}}{\mu}\right)^\frac{N}{s\,p}.
\]
\noindent
This lower bound can be improved, if we consider {\em sign-changing} solutions. This is the content of the next useul result.

\begin{lemma}[Energy doubling]
\label{signchanging}
Assume that $u\in D^{s,p}_0(E)$ is a sign-changing weak solution to \eqref{2.10}
where $\mu>0$ and $E$ is a (possibly unbounded) domain in $\mathbb{R}^N$.\ Then
\begin{equation}
\label{doublings}
\|u\|^{p^*_s}_{L^{p^*_s}(E)}\geq 2\left(\frac{\mathcal{S}_{p,s}}{\mu}\right)^\frac{N}{sp},\quad
[u]^p_{D^{s,p}(\mathbb{R}^N)}\geq 2\,\mu\left(\frac{\mathcal{S}_{p,s}}{\mu}\right)^\frac{N}{sp},\quad
I_\infty(u)\geq 2\,\mu\,\frac{s}{N}\,\left(\frac{\mathcal{S}_{p,s}}{\mu}\right)^\frac{N}{sp}.
\end{equation}
\end{lemma}
\begin{proof}
For $p=2$, see \cite[Lemma 2.5]{sss}.\ In the general case, the heuristic idea is to exploit the fact that $u_\pm:=\max\{\pm u,0\}\in D^{s,p}_0(E)\setminus \{0\}$ are both positive subsolutions of \eqref{2.10}. Thus the above universal estimates hold for both of them separately.
More precisely, it is readily seen that for a.e.\ $(x,y)\in\mathbb{R}^{2N}$
the following inequalities hold
\begin{align*}
&J_p(u(x)-u(y))(u_+(x)-u_+(y))\geq |u_+(x)-u_+(y)|^p, \\
&J_p(u(x)-u(y))(u_-(y)-u_-(x))\geq |u_-(x)-u_-(y)|^p.
\end{align*}
Then, testing equation \eqref{2.10} by $u_+$ (respectively\ $-u_-$) yields
\begin{align*}
[u_+]^p_{D^{s,p}(\mathbb{R}^N)} &\leq \int_{\mathbb{R}^{2N}}\frac{J_p(u(x)-u(y))\,(u_+(x)-u_+(y))}{|x-y|^{N+sp}}dxdy=\mu\int_{E}{(u_+)}^{p^*_s}dx, \\
[u_-]^p_{D^{s,p}(\mathbb{R}^N)} &\leq\int_{\mathbb{R}^{2N}}\frac{J_p(u(x)-u(y))\,(u_-(y)-u_-(x))}{|x-y|^{N+sp}}dxdy=\mu\int_{E}{(u_-)}^{p^*_s}dx.
\end{align*}
As before, we can combine these equalities with $\mathcal{S}_{p,s}\,\|u_\pm\|_{L^{p^*_s}(E)}^p\leq [u_\pm]^p_{D^{s,p}(\mathbb{R}^N)}$ to get
\[
\|u_\pm\|^{p^*_s}_{L^{p^*_s}(E)}\geq \left(\frac{\mathcal{S}_{p,s}}{\mu}\right)^{N/sp}.
\]
By summing up these two inequalities, we get the first estimate in \eqref{doublings}. The second one is then obtained by observing that from the equation we have
\[
[u]^p_{D^{s,p}(\mathbb{R}^N)}=\mu\,\|u\|^{p^*_s}_{L^{p^*_s}(E)}.
\]
Finally, for the third estimate in \eqref{doublings} we observe that from the previous identity
\[
I_\infty(u)=\frac{1}{p}\,[u]^p_{D^{s,p}(\mathbb{R}^N)}- \frac{\mu}{p^*_s}\,\|u\|^{p^*_s}_{L^{p^*_s}(E)}=
\mu\,\left(\frac{1}{p}-\frac{1}{p^*_s}\right)\,\|u\|^{p^*_s}_{L^{p^*_s}(E)}\geq 2\,\mu\,\frac{s}{N}\,\left(\frac{\mathcal{S}_{p,s}}{\mu}\right)^{N/sp},
\]
which completes the proof.
\end{proof}


\medskip
\section{Proof of Theorem \ref{Theorem3}}

\noindent
We divide the proof into five steps.
\vskip4pt
\noindent
{\bf $\RHD$ Step 1.} We first observe that the Palais-Smale sequence $\{u_n\}_{n\in\mathbb{N}}$ is bounded in $D_0^{s,p}(\Omega)$. In fact, by hypothesis we have
\begin{equation}
\label{In}
I(u_n) =\frac{1}{p}\,[u_n]^p_{D^{s,p}(\mathbb{R}^N)}+\frac{1}{p}\,\int_\Omega a\,|u_n|^pdx-\frac{\mu}{p^*_s}\,\int_\Omega|u_n|^{p_s^*}\,dx=c+o_n(1), \\
\end{equation}
and
\[
[u_n]^p_{D^{s,p}(\mathbb{R}^N)}
+\int_\Omega a\,|u_n|^{p}dx-\mu\,\int_\Omega |u_n|^{p_s^\ast}\,dx=\langle I'(u_n),u_n\rangle=o_n(1)\,[u_n]_{D^{s,p}(\mathbb{R}^N)}, 
\]
as $n\to\infty$, which yields
\begin{equation}\label{I'}
\mu\left(\frac{1}{p}-\frac{1}{p^*_s}\right)\int_\Omega|u_n|^{p^*_s}dx=I(u_n)-\frac{1}{p}\langle I'(u_n),u_n\rangle\leq c+1+o_n(1)\,[u_n]_{D^{s,p}(\mathbb{R}^N)}.
\end{equation}
In turn, by H\"older inequality and \eqref{I'}, with simple manipulations it follows
\begin{equation}
\label{I''}
\left|\int_\Omega a\,|u_n|^pdx\right| \leq \|a\|_{L^{N/sp}(\Omega)}\left(\int_\Omega |u_n|^{p_s^*}dx\right)^\frac{p}{p^*_s}
\leq C+o_n(1)\,[u_n]_{D^{s,p}(\mathbb{R}^N)},
\end{equation}
where $C>0$ depends on $N,s,p,\mu,c$ and the norm of $a$, but not on $n$.
Whence, from \eqref{In}, \eqref{I'} and \eqref{I''}, we infer, as $n\to\infty$
$$
[u_n]^p_{D^{s,p}(\mathbb{R}^N)}\leq C+o_n(1)[u_n]_{D^{s,p}(\mathbb{R}^N)},
$$
which shows the boundedness in $D^{s,p}_0(\Omega)$.
Hence, passing if necessary to a subsequence, we have
$u_n\rightharpoonup v^0$ in $D_0^{s,p}(\Omega)$ and $u_n\to v^0$ almost everywhere in $\Omega$.
By Lemma \ref{proof4}, it follows that $I'(v_0)=0$ and $u_n^1:=u_n-v^0\in D^{s,p}_0(\Omega)$ is 
a Palais-Smale sequence for $I_\infty$ at level $c-I(v_0)$, and
\[
[u_n^1]^p_{D^{s,p}(\mathbb{R}^N)}=[u_n]^p_{D^{s,p}(\mathbb{R}^N)}-[v_0]^p_{D^{s,p}(\mathbb{R}^N)}+o_n(1),\qquad \mbox{ as } n\to \infty.
\]
\vskip.2cm
\noindent
{\bf $\RHD$ Step 2.} If $u_n^1\to 0$ in $L^{p^*_s}(\mathbb{R}^N)$ up to a subsequence, since $I_\infty'(u_n^1)\to 0$ in $D^{-s,p'}(\Omega)$
we have 
\[
[u_n^1]^p_{D^{s,p}(\mathbb{R}^N)}-\mu\,\int_{\mathbb{R}^N} |u_n^1|^{p_s^\ast}\,dx=\langle I_\infty'(u_n^1),u_n^1\rangle=o_n(1)\,[u_n^1]_{D^{s,p}(\mathbb{R}^N)}.
\]
Since this sequence is bounded in $D^{s,p}_0(\Omega)$, this yields that $[u_n^1]_{D^{s,p}(\mathbb{R}^N)}\to 0$ as $n$ goes to $\infty$, thus
completing the proof. Let us now suppose that $\{u_n^1\}_{n\in{\mathbb N}}$ does not converge to 
$0$ in $L^{p^*_s}(\mathbb{R}^N)$. Then, up to a subsequence, we have
\[
\inf_{n\in\mathbb{N}} \int_{\mathbb{R}^N} |u_n^1|^{p^*_s}\, dx:=\delta_0>0.
\]
We now take $0<\delta<\delta_0$, to be specified later on, and introduce the {\it Levy concentration function}
$$
Q_n(r):=\sup_{\xi\in\mathbb{R}^N}\int_{B_r(\xi)}|u_n^1|^{p^*_s}\,dx,\quad\, r\geq 0, \,\, n\in\mathbb{N}.
$$
For all $n\in\mathbb{N}$, the function $r\mapsto Q_n(r)$ is continuous on $\mathbb{R}_+$ (see Lemma \ref{lm:levy} below).
This and the fact that $Q_n(0)=0$ and $Q_n(\infty)>\delta$ imply the existence of
$\{\lambda_n^1\}_{n\in\mathbb{N}}\subset\mathbb{R}_+$ such that
\[
Q_n(\lambda_n^1)=\sup_{\xi\in\mathbb{R}^N}\int_{B_{\lambda_n^1}(\xi)}|u_n^1|^{p^*_s}\,dx=\delta.
\]
Moreover, since $|u_n|^{p^*_s}$ vanishes outside $\Omega$, still by Lemma \ref{lm:levy} we know that
\[
\delta=Q_n(\lambda_n^1)=\int_{B_{\lambda_n^1}(z_n^1)}|u_n^1|^{p^*_s}\,dx,\quad \mbox{ for some } z_n^1\in\{x\in\mathbb{R}^N\, :\, \mathrm{dist}(x,\Omega)\le \lambda^1_n\}.
\]
Before proceeding further, we record the following observation: since if $\lambda^1_n\ge \mathrm{diam} (\Omega)$, then 
\[
Q_n(\lambda^1_n)=\sup_{\xi\in\mathbb{R}^N}\int_{B_{\lambda_n^1}(\xi)}|u_n^1|^{p^*_s}\,dx=\int_{\Omega}|u_n^1|^{p^*_s}\,dx>\delta=Q_n(\lambda^1_n),
\]
we obtain that the sequence $\{\lambda^1_n\}_{n\in\mathbb{N}}$ is bounded. This in turn implies that $\{z_n^1\}_{n\in\mathbb{N}}$ is bounded as well, by construction.  We consider now the sequence $v_n^1:\Omega_n\to \mathbb{R}$ defined by
$$
v_n^1(x):=(\lambda_n^1)^{\frac{N-s\,p}{p}}\,u_n^1(\lambda_n^1\,x+z_n^1),\qquad \Omega_n:=\frac{1}{\lambda_n^1}(\Omega-z_n^1)
$$
In light of Lemma~\ref{proof7} the sequence $\{v_n^1\}_{n\in\mathbb{N}}$ is bounded in $D^{s,p}_0(\mathbb{R}^N)$ (because so is $\{u_n^1\}_{n\in\mathbb{N}}$) and thus
we can assume that 
\[
v_n^1\rightharpoonup v^1\ \mbox{ in } D^{s,p}_0(\mathbb{R}^N),\qquad v_n^1\to v^1
\mbox{ in } L_{{\rm loc}}^\sigma(\mathbb{R}^N)\ \mbox{ for every } \sigma\in [1,p^*_s),
\] 
and 
\[
v_n^1\to v^1,\quad \mbox{ a.e. on }\mathbb{R}^N,
\]
up to a subsequence. Observe also that
\begin{equation}
\label{delta}
\delta=\int_{B_{\lambda_n^1}(z_n^1)} |u_n^1|^{p^*_s}\,dx=\int_{B_1(0)}|v_n^1|^{p^*_s}\,dx=\sup_{z\in\mathbb{R}^N}\int_{B_1(z)}|v_n^1|^{p^*_s}\,dx,
\end{equation}
and this in turn implies that 
\begin{equation}
\label{porcanna}
|B_{\lambda_n^1}(z_n^1)\cap \Omega|>0.
\end{equation}
\vskip4pt
\noindent
{\bf $\RHD$ Step 3.} 
The argument that we exploit in this step is substantially
different from the argument originally devised by Struwe in \cite{struwe}, requiring a delicate 
extension procedure on the sequence of approximate solutions. We rather follow a related argument contained in \cite{clapp}.
\vskip.2cm
We claim that the limit $v_1$ found at the previous Step 2 is $v_1\not=0$. Suppose by contradiction that $v_1=0$ almost everywhere. Then, we would have that $v_n^1\to0$
in $L_{{\rm loc}}^\sigma(\mathbb{R}^N)$, for every $\sigma\in [1,p^*_s)$.
Let $h\in C^\infty_0(\mathbb{R}^N)$ be positive and such that 
\begin{equation}
\label{location}
{\rm supp}(h)\subset B_1(z)\subset B_{3/2}(0),\quad\text{for an arbitrary $z\in B_{1/2}(0)$.}
\end{equation}
We now recall that for functions in $D^{s,p}_0(B_{3/2}(0))$ the following Sobolev inequality holds (see \cite[Proposition 2.3]{BP} with the choices $r=3/2$ and $R=2$ there)
\begin{equation}
\label{apri!}
\left(\int_{B_{3/2}(0)} |u|^{p^*_s}\, dx\right)^\frac{p}{p^*_s}\le \mathcal{T}\,[u]^p_{D^{s,p}(B_2(0))},
\end{equation}
for a constant $\mathcal{T}=\mathcal{T}(N,s,p)>0$.
By the H\"{o}lder inequality and \eqref{apri!},
since $h\, v_n^1\in D^{s,p}_0(B_{3/2}(0))$, it follows that
\begin{equation}
\begin{split}
\label{stimasupp}
\int_{\mathbb{R}^N}h^p\,|v_n^1|^{p^*_s}\,dx & \leq \left(\int_{B_1(z)}|v_n^1|^{p^*_s}\,dx\right)^{\frac{s\,p}{N}}\left(\int_{B_{3/2}(0)}\left(h\,|v_n^1|\right)^{p^*_s}\,dx\right)^\frac{p}{p^*_s} \\
&\leq  \mathcal{T}\,\left(\int_{B_1(z)}|v_n^1|^{p^*_s}\,dx\right)^{\frac{s\,p}{N}}\,\big[ h\,v_n^1\,\big]^p_{D^{s,p}(B_2(0))},
\end{split}
\end{equation}
for some positive constant $\mathcal{T}$ depending only on $N,s,p$. We now observe that by the very definition of $I'_\infty$
\begin{equation*}
\begin{split}
\int_{\mathbb{R}^{2\,N}}&\frac{J_p(v_n^1(x)-v_n^1(y))(\varphi(x)-\varphi(y))}{|x-y|^{N+s\,p}}\,dx\,dy\\
&=\mu\int_{\mathbb{R}^N} |v_n^1|^{p^*_s-2} v_n^1\,\varphi \,dx+\langle I'_\infty(v_n^1),\varphi\rangle, \,\,\quad \mbox{for any }\varphi\in D^{s,p}_0(\Omega_n).
\end{split}
\end{equation*}
Then, by applying Proposition~\ref{prop:caccioloc} for every $n\in\mathbb{N}$ with the choices  
\[
\Omega:=\Omega_n, \quad \Omega':=B_2(0),\quad
u:=v^1_n,\quad \psi:=h,\quad F:=\mu\,|v^1_n|^{p^*_s-2}\,v^1_n+I'_\infty(v^1_n),
\]
we get
\begin{equation}
\label{casto}
\begin{split}
\int_{B_2(0)\times B_2(0)} &\frac{\big|v^1_n(x)\,h(x)-v^1_n(y)\,h(y)\big|^{p}\,}{|x-y|^{N+s\,p}}\, dx\, dy\\
&\le \mathcal{C}\,\int_{B_2(0)\times B_2(0)} \frac{|h(x)-h(y)|^p}{|x-y|^{N+s\,p}}\,\Big(|v^1_n(x)|^p+|v^1_n(y)|^p\Big)\, dx\,dy\\
&+\mathcal{C}\,\left(\sup_{y\in B_{3/2}(0)} \int_{\mathbb{R}^N\setminus B_2(0)} \frac{|v^1_n(x)|^{p-1}}{|x-y|^{N+s\,p}}\,dx\right)\, \int_{B_{3/2}(0)} |v^1_n|\, h^p\, dx\\
&+\mathcal{C}\,\int_{B_{3/2}(0)} h^p\,|v^1_n|^{p^*_s}\,dx +\mathcal{C}\,\Big|\langle I'_\infty(v^1_n),v^1_n\,h^p\rangle\Big|.
\end{split}
\end{equation}
Observe that thanks to \eqref{delta}, we know that $B_2(0)\cap \Omega_n$ in a non-empty open set.
We proceed to estimate the terms on the right-hand side of \eqref{casto}. For the first term on the right-hand side, we have
\[
\begin{split}
\int_{B_2(0)\times B_2(0)} &\frac{|h(x)-h(y)|^p}{|x-y|^{N+s\,p}}\,\Big(|v^1_n(x)|^p+|v^1_n(y)|^p\Big)\, dx\,dy\\
&\le \|\nabla h\|_{L^\infty}^p\,\int_{B_2(0)}\left(\int_{B_2(0)} \frac{dy}{|x-y|^{N+s\,p-p}}\right)\,|v^1_n(x)|^p\,dx\\
&+\|\nabla h\|_{L^\infty}^p\,\int_{B_2(0)}\left(\int_{B_2(0)} \frac{dx}{|x-y|^{N+s\,p-p}}\right)\,|v^1_n(y)|^p\,dy=o_n(1),
\end{split}
\]
thanks to the local strong $L^p$ convergence to $0$ of $\{v^1_n\}_{n\in{\mathbb N}}$. For the second term on 
the right-hand side of \eqref{casto}, we observe that for the same reason we have
\[
\int_{B_{3/2}(0)} |v^1_n|\, h^p\, dx=o_n(1),
\]
while by H\"older inequality, for every $y\in B_{3/2}(0)$ we get
\[
\begin{split}
\int_{\mathbb{R}^N\setminus B_2(0)} \frac{|v^1_n(x)|^{p-1}}{|x-y|^{N+s\,p}}\,dx&\le \left(\int_{\mathbb{R}^N} |v^1_n|^{p^*_s}\,dx\right)^\frac{p-1}{p^*_s}\\
&\times\left(\int_{\mathbb{R}^N\setminus B_2(0)} |x-y|^{-(N+s\,p)\,\frac{p^*_s}{p^*_s-p+1}}\,dx\right)^\frac{p^*_s-p+1}{p^*_s},
\end{split}
\]
which is uniformly bounded. For the third term, by using inequality \eqref{stimasupp}, and recalling \eqref{delta} and \eqref{location}, we have
\[
\int_{B_{3/2}(0)} h^p\,|v^1_n|^{p^*_s}\,dx\le \mathcal{T}\,\left(\int_{B_1(z)}|v_n^1|^{p^*_s}\,dx\right)^{\frac{s\,p}{N}}\,\big[ h\,v_n^1\big]^p_{D^{s,p}(B_2(0))}\leq 
\mathcal{T}\,\delta^{\frac{s\,p}{N}}\,\big[ h\,v_n^1\big]^p_{D^{s,p}(B_2(0))}
\]
For the last term, since $I_\infty'(u_n^1)\to 0$,
we learn from ($a_2$) of Lemma~\ref{proof7} that 
\[
\sup_{\varphi\in  D^{s,p}_0(\Omega_n)}\Big|\Big\langle I'_\infty(v^1_n),\frac{\varphi}{[\varphi]_{D^{s,p}(\mathbb{R}^N)}}\Big\rangle\Big|=o_n(1),
\] 
thus in particular $|\langle I'_\infty(v_n^1),h^p\,v^1_n\rangle|=o_n(1),$
since the sequence $\{h^p\,v_n^1\}_{n\in\mathbb{N}}$ is bounded in $D^{s,p}_0(\Omega_n)$ in view of Lemma \ref{Bddlemma} (recall that $\{v^1_n\}_{n\in\mathbb{N}}$ is bounded in $D^{s,p}_0(\Omega_n)$).
By introducing the previous estimates in \eqref{casto}, we thus get
\[
\begin{split}
\big[ h\,v_n^1\big]^p_{D^{s,p}(B_2(0))}\le \mathcal{C}\,\mathcal{T}\,\delta^\frac{s\,p}{N}\,\big[ h\,v_n^1\big]^p_{D^{s,p}(B_2(0))}+o_n(1),
\end{split}
\]
where we recall that $\mathcal{C}$ is the constant appearing in the Caccioppoli inequality of Proposition \ref{prop:caccioloc} and this depends on $p$ only.
By choosing\footnote{Observe in particular that $\delta$ depends on $N,s,p,\mu$ and $\delta_0$ only. Also observe that we can always suppose $\delta_0<1$.} 
\[
\delta=\min\left\{\frac{1}{2\,\mathcal{C}\,\mathcal{T}},\,\frac{\delta_0}{2}\right\}^\frac{N}{s\,p},
\] 
from the previous inequalities we obtain 
\[
[h\,v_n^1]_{D^{s,p}(B(0,2))}=o_n(1),\qquad \mbox{ as } n\to\infty.
\] 
By using again the Sobolev inequality \eqref{apri!}, this in turn implies
$$
\int_{B_{3/2}(0)}\left(h\,|v_n^1|\right)^{p^*_s}\,dx=o_n(1).
$$
By arbitrariness of $h\in C^\infty_0(B_1(z))$, we obtain that $\{v_n^1\}_{n\in\mathbb{N}}$ converges to zero in $L^{p^*_s}_{{\rm loc}}(B_1(z))$.
Finally, taking into account the condition \eqref{location} and the arbitrariness of $z\in B_{1/2}(0)$, we obtain
that $\{v_n^1\}_{n\in\mathbb{N}}$ converges to zero in $L^{p^*_s}(B_1(0))$,
which contradicts \eqref{delta}. Hence, $v^1\not=0$.
\vskip6pt
\noindent
{\bf $\RHD$ Step 4.} We have already seen in {\bf Step 2} that the sequences $\{z_n^1\}_{n\in\mathbb{N}}$ and $\{\lambda^1_n\}_{n\in\mathbb{N}}$ are bounded, thus we may assume that $z_n^1\to z_0^1\in \mathbb{R}^N$ and $\lambda_n^1\to\lambda_0^1\geq0$. If $\lambda_0^1>0$ then as a consequence
of the fact that $u_n^1\rightharpoonup 0$ in $D_0^{s,p}(\Omega)$, we have $v_n^1\rightharpoonup 0$ in $D^{s,p}_0(\mathbb{R}^N)$
by Lemma~\ref{weak-c} and this is impossible by the previous {\bf Step 3}. Thus $\lambda_n^1\to 0$ and by construction this implies 
\[
\lim_{n\to\infty} \mathrm{dist}(z_n^1,\partial\Omega)=0\qquad \mbox{ and }\qquad z_0^1\in\overline\Omega. 
\]
We now distinguish two cases: 
\[
 \mbox{ either }\quad \lim_{n\to\infty}\frac{1}{\lambda_n^1}\mathrm{dist}(z_n^1,\partial \Omega)=\infty\quad \mbox{ or }\quad \liminf_{n\to\infty}\frac{1}{\lambda_n^1}\mathrm{dist}(z_n^1,\partial \Omega)<\infty.
\]
In the first case, by Lemma \ref{proof5} we have $I_\infty'(v^1)=0$ so that
\[
(- \Delta)_p^s\, v^1  =  \mu\,|v^1|^{p_s^\ast - 2}\,v^1,\qquad \mbox{ in } \mathbb{R}^N.
\]
Moreover, by recalling \eqref{porcanna}, we obtain that 
\[
z_n^1\in \{x\in \Omega\, :\, \mathrm{dist}(x,\partial\Omega)\ge \lambda^1_n\},
\]
for $n$ sufficiently large.
In the second case, by Lemma \ref{proof6} we would have $v_1\in D^{s,p}_0(\mathcal{H})$ for a suitable half-space $\mathcal{H}$ and
\begin{equation*}
\begin{cases}
(- \Delta)_p^s\, v^1  =  \mu\,|v^1|^{p_s^\ast - 2}\,v^1, & \text{ in }\mathcal{H}, \\
v^1=0, & \text{ in } \mathbb{R}^N\setminus \mathcal{H}.
\end{cases}
\end{equation*}
On account of Assumption {\bf (NA)}, this case is ruled out.
\par
We set $\varrho^1_n=\mathrm{dist}(z_n^1,\partial\Omega)/2$ and take $\zeta\in C^\infty_0(B_2(0))$ a standard cut-off function, such that $\zeta\equiv 1$ on $B_1(0)$. We consider the sequence
\[
u_n^2(z):=u_n^1(z)-(\lambda_n^1)^{\frac{s\,p-N}{p}}\,v^1
\left(\frac{z-z_n^1}{\lambda_n^1}\right)\,\zeta\left(\frac{z-z_n^1}{\varrho_n^1}\right)\in D^{s,p}_0(\Omega),
\]
by construction we have that $\lambda_n^1/\varrho_n^1$ converges to $0$, as $n$ goes to $\infty$. Thus Lemma \ref{proof5} assures that $\{u_n^2\}_{n\in\mathbb{N}}$ is a Palais-Smale sequence for $I_\infty$ at the energy level $c-I(v^0)-I_\infty(v^1)$ such that 
\begin{equation*}
 [u_n^2]^p_{D^{s,p}(\mathbb{R}^N)}=[u_n]^p_{D^{s,p}(\mathbb{R}^N)}-[v^0]^p_{D^{s,p}(\mathbb{R}^N)}-[v^1]^p_{D^{s,p}(\mathbb{R}^N)}+o_n(1).
\end{equation*}
\vskip4pt
\noindent
{\bf $\RHD$ Step 5.} We can iterate the previous construction to cook-up a sequence $\{v^k\}_{k\in\mathbb{N}}$ of critical points of $I_\infty$ and, for every $k\in \mathbb{N}$, sequences $\{z_n^k\}_{n\in\mathbb{N}}$, $\{\lambda^k_n\}_{n\in\mathbb{N}}$, $\{\varrho_n^k\}_{n\in\mathbb{N}}$ and $\{u^k_n\}_{n\in\mathbb{N}}\subset D^{s,p}_0(\Omega)$ with
\[
u_n^k(z):=u^1_n(z)-\sum_{i=1}^{k-1}(\lambda_n^i)^{\frac{s\,p-N}{p}}v^i
\Big(\frac{z-z_n^i}{\lambda_n^i}\Big)\,\zeta\left(\frac{z-z_n^i}{\varrho_n^i}\right),
\]
where $\zeta$ is the same cut-off function as above.
By construction, we have that $\{u^k_n\}_{n\in\mathbb{N}}$
is a Palais-Smale sequence for $I_\infty$ at the energy level
$$
c-I(v^0)-\sum_{i=1}^{k-1} I_\infty(v^i),
$$
and, furthermore,
$$
[u_n^k]^p_{D^{s,p}(\mathbb{R}^N)}=[u_n]^p_{D^{s,p}(\mathbb{R}^N)}-\sum_{i=0}^{k-1}[v^i]^p_{D^{s,p}(\mathbb{R}^N)}+o_n(1).
$$
Observe that each $v^1,\dots, v^{k-1}$ is a critical point of $I_\infty$, thus from \eqref{snorme} we get
\[
[u_n^k]^p_{D^{s,p}(\mathbb{R}^N)}\le [u_n]^p_{D^{s,p}(\mathbb{R}^N)}-[v^0]^p_{D^{s,p}(\mathbb{R}^N)}-(k-1)\,\mu\,\left(\frac{\mathcal{S}_{p,s}}{\mu}\right)^\frac{N}{s\,p}+o_n(1),
\]
which implies that this iterative construction must stop at some $k_0\in\mathbb{N}$. As at the beginning of {\bf Step 2}, this means that $[u_n^{k_0}]_{D^{s,p}(\mathbb{R}^N)}\to 0$ as $n$ goes to $\infty$. This in turn yields \eqref{decomposizione}, \eqref{decomposizione2} and \eqref{decomposizione3}, as desired. \qed

\vskip8pt
\noindent
In {\bf Step 2} above we used the following result, which is well-known. We record its proof for the sake of completeness.
\begin{lemma}
\label{lm:levy}
Let $f\in L^1(\mathbb{R}^N)$, then its Levy concentration function 
$$
Q_f(r):=\sup_{\xi\in\mathbb{R}^N}\int_{B_r(\xi)}|f|\,dx,\quad\, r\geq 0,
$$
is a continuous function. If $f\equiv 0$ outside a bounded set $K$ with smooth boundary, then for every $r\ge 0$ the supremum in the definition of $Q_f(r)$ is actually a maximum. More precisely, we have
\[
Q_f(r):=\max_{\xi\in K_r}\int_{B_r(\xi)}|f|\,dx,\qquad \mbox{ with } K_r=\{x\in\mathbb{R}^N\, :\, \mathrm{dist}(x,\overline K)\le r\}.
\]
\end{lemma}
\begin{proof}
The function $Q_f$ is monotone non decreasing. Observe that for every $\xi\in \mathbb{R}^N$, the function
\[
r\mapsto \int_{B_r(\xi)}|f|\,dx
\]
is continuous, then $Q_f$ is lower semicontinuous as a supremum of continuous functions. Let us suppose that there exists $r_0>0$ such that 
\[
\ell^+:=\lim_{r\to r_0^+} Q_f(r)\not=\lim_{r\to r_0^-} Q_f(r)=:\ell^-. 
\]
By monotonicity and lower semicontinuity of $Q_f$, this means that $\ell^+>\ell^-= Q_f(r_0)$. Let us set $\varepsilon=\ell^+-Q_f(r_0)$, then for every $r>r_0$ we have
\[
Q_f(r)-Q_f(r_0)\ge \varepsilon.
\]
By definition of $Q_f$, we can then choose $\xi_0=\xi_0(\varepsilon,r)\in\mathbb{R}^N$ such that
\[
\frac{\varepsilon}{2}\le \int_{B_{r}(\xi_0)}|f|\,dx-\int_{B_{r_0}(\xi_0)} |f|\,dx=\int_{B_{r}(\xi_0)\setminus B_{r_0}(\xi_0)} |f|\,dx .
\]
Since the measure of the annulus $B_{r}(\xi_0)\setminus B_{r_0}(\xi_0)$ converges to $0$ as $r\searrow r_0$, this gives the desired contradiction.
Let us now assume that $f= 0$ almost everywhere in $\mathbb{R}^N\setminus K$. For every $r>0$ the function 
\[
\xi\mapsto \int_{B_r(\xi)}|f|\,dx,
\]
is continuous and it vanishes if $B_r(\xi)\subset \mathbb{R}^N\setminus K$. This happens if $\mathrm{dist}(\xi,K)>r$ and we conclude the proof.
\end{proof}

\begin{remark}
	We observe that if the level $c$ satisfies
	\begin{equation}
	\label{staisotto?}
	c<2\,\frac{s}{N}\,\mu\,\left(\frac{\mathcal{S}_{p,s}}{\mu}\right)^{N/sp}.
	\end{equation}
	then $k$ in Theorem \ref{Theorem3} is either $0$ ({\it compactness holds}) or $k=1$ ({\it compactness fails}). In the second case, the unique function $v^1$ must have constant sign and be different from $0$ almost everywhere. Indeed, let us assume \eqref{staisotto?} and
	observe that $I(v^0)\ge 0$, since $v_0$ is a critical point of $I$.
	If we suppose that $v^1$ is sign-changing, from Lemma \ref{signchanging} and the decomposition \eqref{decomposizione3} we would get
\[
c=I(v^0)+I_\infty(v^1)\ge 2\,\mu\,\frac{s}{N}\,\left(\frac{\mathcal{S}_{p,s}}{\mu}\right)^{N/sp},
\]
thus contradicting \eqref{staisotto?}. This implies that $v^1$ has constant sign and we can conclude that $v^1\not =0$ almost everywhere, thanks to Proposition \ref{minimumP}.
\end{remark}

We say that  $\{u_n\}_{n\in\mathbb{N}}\subset D_0^{s,p}(\Omega)$ is a {\em Palais-Smale sequence with sign} for $I$ 
at level $c$ if it is a Palais-Smale sequence and 
\[
\lim_{n\to\infty} \|(u_n)_-\|_{L^{p^*_s}(\Omega)}=0. 
\]
\noindent
With minor modifications in the proof of Theorem~\ref{Theorem3}, we can get the following variant for Palais-Smale sequences with sign. We leave the details to the reader.
\begin{theorem}
We assume hypothesis {\bf (NA)}.  Let $1<p<\infty$ and $s\in(0,1)$ be such that $N>s\,p$. Let $\Omega\subset\mathbb{R}^N$ be an open bounded set with smooth boundary. Let $\{u_n\}_{n\in\mathbb{N}}\subset D_0^{s,p}(\Omega)$ be a Palais-Smale sequence with sign for the functional $I$ defined in \eqref{I}  at level $c$. 
\par
Then there exist: 
\begin{itemize}
\item a (possibly trivial) non-negative solution $v_0\in D_0^{s,p}(\Omega)$ of
\[
(- \Delta)_p^s\, u +a\,|u|^{p-2}\,u = \mu\,u^{p_s^\ast - 1},\qquad \mbox{ in } \Omega, 
\]
\item a number $k\in\mathbb{N}$ and $v^1,v^2\cdots,v^k\in D^{s,p}({\mathbb R}^N)\setminus\{0\}$ positive solutions of
\[
(- \Delta)_p^s\, u  =  \mu\,u^{p_s^\ast - 1},\qquad \text{ in } {\mathbb R}^N.
\]
\vskip.2cm
\item a sequence of positive real numbers $\{\lambda_n^i\}_{n\in\mathbb{N}}\subset \mathbb{R}_+$ with $\lambda_n^i\to 0$ and a sequence of points $\{z_n^i\}_{n\in\mathbb{N}}\subset \{x\in\Omega\, :\, \mathrm{dist}(x,\partial\Omega)\ge \lambda^i_n\}$, for $i=1,\dots,k$;
\end{itemize}
\vskip.2cm
such that, up to a subsequence,  conclusions \eqref{decomposizione}, \eqref{decomposizione2} and \eqref{decomposizione3} follow.
\end{theorem}
	The positivity
	of the limiting profiles $v^1,\dots,v^k$ in the result above can be obtained by appealing again to the minimum principle of Proposition~\ref{minimumP}.

\section{Radial case}

\subsection{Improved embeddings for radial functions}

In the proof of Theorem \ref{radial-case}, we need the following embedding result for the space $D^{s,p}_{0,\mathrm{rad}}(B)$. In what follows, by $K\Subset E$ we mean that $K$ is an open bounded set with compact closure contained in $E$.
\begin{proposition}[Compact embeddings]
\label{prop:radembedding}
Let $1<p<\infty$ and $s\in(0,1)$, we set
\[
p^\#_s=\left\{\begin{array}{rc}
\dfrac{p}{1-s\,p},& \mbox{ if }s\,p<1,\\
+\infty,& \mbox{ if }s\,p\ge 1.
\end{array}
\right.
\]
Then we have the compact embedding  
\[
D^{s,p}_{0,\rm{rad}}(B_R)\hookrightarrow L^q(K),
\]
for every $1\le q<p^\#_s$ and every $K\Subset\mathbb{R}^N\setminus\{0\}$.
\end{proposition}
\begin{proof}
Let us start with the case $s\,p>1$. We remark that we already know that the embedding $D^{s,p}_{0}(B_R)\hookrightarrow L^p_{\rm loc}(\mathbb{R}^N)$ is compact (for example, see \cite[Theorem 2.7]{bralinpar}). A simple interpolation argument permits to infer the desired conclusion. Indeed, let us take $q> p$, a set $K\Subset\mathbb{R}^N\setminus\{0\}$, for every $u\in D^{s,p}_{0,\rm{rad}}(B_R)$ by using Lemma \ref{lm:radial} and \eqref{esplosione} we obtain
\[
\begin{split}
\int_{K} |u|^q\,dx&=\int_{K} \left(|x|^\frac{N-s\,p}{p}\,|u|\right)^{q-p}\,|u|^p\,|x|^{-\frac{N-s\,p}{p}\,(q-p)}\,dx\\
&\le C_K\,[u]^{q-p}_{D^{s,p}(B_R)}\,\int_K |u|^p\,dx.
\end{split}
\] 
Thanks to this we can get the desired conclusion. 
\par
As far as the case $sp\leq 1$ is concerned,
we still use that $D^{s,p}_{0}(B_R)$ compactly embeds into $L^p(B_R)$ and then 
the assertion follows by Lemma \ref{lm:radial} jointly with a standard interpolation argument in Lebesgue spaces.
\end{proof}
\begin{remark}[The exponent $p^\#_s$]
\label{oss:boh!}
We observe that $p^\#_s$ coincides with the one-dimensional Sobolev exponent. In the case $s\,p<1$ it is not possible to go beyond this exponent in Proposition \ref{prop:radembedding}. Indeed, for $s\,p<1$ it is not difficult to construct a bounded sequence $\{u_n\}_{n\in\mathbb{N}}\subset D^{s,p}_{0,{\rm rad}}(B_1(0))$ such that for a suitable compact set $K\subset \mathbb{R}^N\setminus\{0\}$ we have
\[
\lim_{n\to\infty}\|u_n\|_{L^q(K)}=\infty,\qquad \mbox{ for }p^\#_s<q\le \infty.
\]
Let us consider the spherical shells
\[
A_n=\{x\in\mathbb{R}^N\, :\, 1-r_n<|x|<1\},\qquad \mbox{ with } r_n=n^{-\frac{p}{1-s\,p}}.
\]
If we denote by $1_E$ the characteristic function of a set $E$, we observe that the functions $u_n=n\,1_{A_n}$ belong to $D^{s,p}_{0,\rm{rad}}(B_1(0))$. Indeed, if $P(E)$ denotes the perimeter of a smooth set $E\subset\mathbb{R}^N$, we have
\[
[u_n]^p_{D^{s,p}(\mathbb{R}^N)}=n^{p}\,[1_{A_n}]_{D^{s\,p,1}(\mathbb{R}^N)}\le C\,n^{p}\, |A_n|^{1-s\,p}\, P(A_n)^{s\,p},
\]
where the last inequality is \cite[Corollary 4.4]{bralinpar}. It is not difficult to see that 
\[
|A_n|^{1-s\,p}\, P(A_n)^{s\,p}\simeq r_n^{(1-s\,p)}=n^{-p},
\]
which implies that
\[
[u_n]^p_{D^{s,p}(B_1(0))}\le[u_n]^p_{D^{s,p}(\mathbb{R}^N)}\le C.
\]
On the other hand, for $q> p/(1-s\,p)$ we have
\[
\|u_n\|^{q}_{L^{q}(\mathbb{R}^N)}=n^{q}\,|A_n|\simeq n^{q}\,r_n=n^{q-\frac{p}{1-s\,p}},
\]
which diverges. 
\par
We also point out that the very same example shows that in the limit case $q=p^\#_s$ the embedding is continuous, but not compact.
\end{remark}
The previous result was based on the following Radial Lemma for fractional Sobolev spaces. We give the proof for the reader's convenience.
For more general results valid in Besov and Triebel spaces, we refer the reader to \cite{SSV} and \cite[Chapter 6]{triebel2}.
\begin{lemma}[A nonlocal Radial Lemma]
\label{lm:radial}
Let $1<p<\infty$ and $s\in(0,1)$. Let $B_R$ be the ball centered at the origin with radius $R>0$. Then we have the continuous embeddings:
\begin{itemize}
\item if $s\,p>1$
\[
D^{s,p}_{0,\mathrm{rad}}(B_R)\hookrightarrow L^\infty_{\mathrm{loc}}\left(\mathbb{R}^N\setminus\{0\}; |x|^\frac{N-s\,p}{p}\right);
\]
\item if $s\,p<1$
\[
D^{s,p}_{0,\mathrm{rad}}(B_R)\hookrightarrow L^\frac{p}{1-s\,p}_{{\rm loc}}(\mathbb{R}^N\setminus \{0\});
\]
\item if $s\,p=1$
\[
D^{s,p}_{0,\mathrm{rad}}(B_R)\hookrightarrow L^t_{{\rm loc}}(\mathbb{R}^N\setminus \{0\}),\qquad \mbox{ for every }1\le t<\infty.
\]
\end{itemize}
\end{lemma}
\begin{proof}
We divide the proof in three cases.
\vskip.2cm\noindent
\underline{\it Case $s\,p>1$}. Let $0<\varrho<R$, since $u$ is a radial function we get
	\[
	\int_{\partial B_\varrho} |u|^p\,d\mathcal{H}^{N-1}=N\,\omega_N\, \varrho^{N-1} \, |u(x)|^p,\qquad \mbox{ for } |x|=\varrho.
	\]
	We observe that the integral is well-defined, since $u$ has a trace in $L^p(\partial B_\varrho)$ thanks to the hypothesis $s\,p>1$. We can now use the trace inequality for $D^{s,p}(B_\varrho)$ (see \cite[Section 3.3.3]{triebel}), so to obtain
	\[
	\begin{split}
	|u(x)|^p&=\frac{\varrho^{1-N}}{N\,\omega_N}\,\int_{\partial B_\varrho} |u|^p\,d\mathcal{H}^{N-1}\\
	&\le C\,\frac{\varrho^{1-N}}{N\,\omega_N}\, \varrho^{s\,p-1}\,\left\{[u]^p_{D^{s,p}(B_\varrho)}+\frac{1}{\varrho^{s\,p}} \|u\|^p_{L^p(B_\varrho)}\right\},
	\end{split}
	\]
	for some $C=C(N,p,s)>0$.
	In order to get the desired estimate, it is now sufficient to use Poincar\'e inequality (which again needs $s\,p>1$)
	\[
	\frac{1}{\varrho^{s\,p}} \|u\|^p_{L^p(B_\varrho)}\le \frac{1}{\varrho^{s\,p}} \|u\|^p_{L^p(B_R)}\le C\,\left(\frac{R}{\varrho}\right)^{s\,p}\,[u]^p_{D^{s,p}(B_R)},\qquad 0<\varrho<R.
\]
This gives
\begin{equation}
\label{esplosione}
|u(x)|\le C\,|x|^{-\frac{N-s\,p}{p}}\,\left(\frac{R}{|x|}\right)^s\,[u]_{D^{s,p}(B_R)},\qquad 0<|x|<0,
\end{equation}
for some $C=C(N,s,p)>0$. Observe that inequality \eqref{esplosione} holds for $|x|\ge R$ as well, since $u\equiv 0$ on $\mathbb{R}^N\setminus B_R$.
\par
We now take $K\Subset\mathbb{R}^N\setminus\{0\}$.\ Then, there exists $0<R_0<R_1$ such that
\[
K\subset B_{R_1}(0)\setminus B_{R_0}(0).
\]
From \eqref{esplosione} we directly get
\[
\left\||x|^\frac{N-s\,p}{p}\,u\right\|_{L^\infty(K)}\le C\,\left(\frac{R_0}{R_1}\right)^s\,[u]_{D^{s,p}(\mathbb{R}^N)},
\]
which proves the desired embedding.
\vskip.2cm\noindent
\underline{\it Case $s\,p<1$.} Let $u\in D^{s,p}_{0,{\mathrm{rad}}}(\mathbb{R}^N)$. We first show that for every $0<R_0<R_1$ we have (with a slight abuse of notation)
\begin{equation}
\label{N1}
[u]^p_{D^{s,p}(\mathbb{R}^N)}\ge  C\,\int_{R_0}^{R_1} \int_{R_0}^{R_1} \frac{|u(r)-u(\varrho)|^p}{|\varrho-r|^{1+s\,p}}\,d\varrho\,dr,
	\end{equation}
for some $C=C(N,s,p,R_0,R_1)>0$. Indeed, by arguing as in \cite[Lemma B.2]{BF}, we have
	\[
	[u]^p_{D^{s,p}(\mathbb{R}^N)}=C\int_0^\infty \int_0^\infty |u(r)-u(\varrho)|^p\,\varrho^{N-1}\,r^{N-1}\,\Phi(\varrho,r)\,d\varrho\,dr,
	\]
	where
	\[
	\begin{split}
	\Phi(\varrho,r)&:=\int_{-1}^1\dfrac{(1-t^2)^\frac{N-3}{2}}{(\varrho^2-2\,t\,\varrho\,r+r^2)^\frac{N+s\,p}{2}}\,dt\\
	&\ge \int_{1/2}^1\dfrac{(1-t^2)^\frac{N-3}{2}}{\Big((\varrho-r)^2+2\,\varrho\,r\,(1-t)\Big)^\frac{N+s\,p}{2}}\,dt\\
	&=\frac{1}{|\varrho-r|^{N+s\,p}}\,\int_{1/2}^1\dfrac{(1-t^2)^\frac{N-3}{2}}{\Big(1+2\,\dfrac{\varrho\,r}{(\varrho-r)^2}\,(1-t)\Big)^\frac{N+s\,p}{2}}\,dt.
	\end{split}
	\]
	For $\varrho\not= r$, we make the change of variables
	\[
	2\,\dfrac{\varrho\,r}{(\varrho-r)^2}\,(1-t)=\tau.
	\]
	Then, the previous expression becomes
	\[
	\frac{1}{|\varrho-r|^{1+s\,p}}\,\frac{1}{(2\,\varrho\,r)^\frac{N-1}{2}}\,\int_{0}^\frac{\varrho\,r}{(\varrho-r)^2}\dfrac{\left(2-\dfrac{(\varrho-r)^2}{2\,\varrho\,r}\,\tau\right)^\frac{N-3}{2}\,\tau^\frac{N-3}{2}}{(1+\tau)^\frac{N+s\,p}{2}}\,d\tau.
	\]
	For every $0<R_0<R_1$ we thus obtain
	\begin{equation}
	\label{quasi}
	\begin{split}
	[u&]^p_{D^{s,p}(\mathbb{R}^N)}\ge  C\,\int_{R_0}^{R_1} \int_{R_0}^{R_1} |u(r)-u(\varrho)|^p\,\varrho^{N-1}\,r^{N-1}\,\Phi(\varrho,r)\,d\varrho\,dr\\
	&\ge C\int_{R_0}^{R_1} \int_{R_0}^{R_1} \frac{|u(r)-u(\varrho)|^p}{|\varrho-r|^{1+s\,p}}\left[(\varrho\,r)^\frac{N-1}{2}\int_{0}^\frac{\varrho\,r}{(\varrho-r)^2}\dfrac{\Big(2-\dfrac{(\varrho-r)^2}{2\,\varrho\,r}\,\tau\Big)^\frac{N-3}{2}\tau^\frac{N-3}{2}}{(1+\tau)^\frac{N+s\,p}{2}}\,d\tau\right]d\varrho\,dr.
\end{split}
\end{equation}
In order to estimate the last integral, we observe that for $R_0\le \varrho \le R_1$ and $R_0\le r\le R_1$ we have
\begin{equation}
\label{relazioniraggiose}
|\varrho-r|\le R_1-R_0\qquad \mbox{ and }\qquad \frac{\varrho\,r}{(\varrho-r)^2}\ge \frac{\varrho\,r}{(R_1-R_0)^2}\ge \left(\frac{R_0}{R_1-R_0}\right)^2=:\alpha.
\end{equation} 
Thus, we proceed as follows (we assume for simplicity $N\ge 3$)
\[
\begin{split}
(\varrho\,r)^\frac{N-1}{2}\,\int_{0}^\frac{\varrho\,r}{(\varrho-r)^2}&\dfrac{\left(2-\dfrac{(\varrho-r)^2}{2\,\varrho\,r}\,\tau\right)^\frac{N-3}{2}\,\tau^\frac{N-3}{2}}{(1+\tau)^\frac{N+s\,p}{2}}\,d\tau\\
&\ge R_0^{N-1}\,\int_{\frac{\alpha}{2}}^{\alpha}\dfrac{\tau^\frac{N-3}{2}}{\left(1+\tau\right)^\frac{N+s\,p}{2}}\,d\tau\\
&\ge R_0^{N-1}\,\left(\frac{\alpha}{2}\right)^\frac{N-3}{2}\,\int_{\frac{\alpha}{2}}^{\alpha}\dfrac{d\tau}{\left(1+\tau\right)^\frac{N+s\,p}{2}}=C_{R_0,R_1}.
\end{split}
\]
By spending this information into \eqref{quasi}, we obtain \eqref{N1}. Observe that on the right-hand side of \eqref{N1} we have the one-dimensional Gagliardo seminorm of the function $u$ on the interval $[R_0,R_1]$. 
By using Sobolev embedding in dimension $1$, we know that 
\begin{equation}
\label{1d}
\int_{R_0}^{R_1} \int_{R_0}^{R_1} \frac{|u(r)-u(\varrho)|^p}{|\varrho-r|^{1+s\,p}}\,d\varrho\,dr+\int_{R_0}^{R_1} |u|^p\,d\varrho\ge \frac{1}{S}\,\left(\int_{R_0}^{R_1} |u|^\frac{p}{1-s\,p}\,d\varrho\right)^{1-s\,p},
	\end{equation}
	for some $S=S(s,p,R_0,R_1)>0$. 
\par
We now prove the claimed embedding. As above we take $K\Subset\mathbb{R}^N\setminus\{0\}$. Then, there exists $0<R_0<R_1$ such that
	\[
	K\subset B_{R_1}(0)\setminus B_{R_0}(0).
	\]
	For $u\in D^{s,p}_{0,\mathrm{rad}}(\mathbb{R}^N)$ we have
	\[
	[u]^p_{D^{s,p}(\mathbb{R}^N)}\ge \mathcal{S}_{p,s}\,\left( \int_{\mathbb{R}^N} |u|^{p^*_s}\,dx\right)^\frac{p}{p^*_s}\ge \mathcal{S}_{p,s}\,|B_{R_1}(0)|^{-s\,p}\,\int_{B_{R_1}(0)\setminus B_{R_0}(0)} |u|^p\,dx,
	\]
	thus we get
	\[
	[u]^p_{D^{s,p}(\mathbb{R}^N)}\ge C\,[u]^p_{D^{s,p}(\mathbb{R}^N)}+C\,\int_{B_{R_1}(0)\setminus B_{R_0}(0)} |u|^p\,dx,
	\]
	for some $C=C(N,s,p,R_1)>0$. We now use polar coordinates, take advantage of the fact that $R_0>0$ and use formula \eqref{N1}. Therefore, we have (the constant $C$ may vary from line to line)
	\[
	\begin{split}
	[u]^p_{D^{s,p}(\mathbb{R}^N)}&\ge C\,[u]^p_{D^{s,p}(\mathbb{R}^N)}+C\,\int_{B_{R_1}(0)\setminus B_{R_0}(0)} |u|^p\,dx\\
	&\ge C\,\int_{R_0}^{R_1} \int_{R_0}^{R_1} \frac{|u(r)-u(\varrho)|^p}{|\varrho-r|^{1+s\,p}}\,d\varrho\,dr+C\,\int_{R_0}^{R_1} |u|^p\,\varrho^{N-1}\,d\varrho\\
	&\ge C\,\int_{R_0}^{R_1} \int_{R_0}^{R_1} \frac{|u(r)-u(\varrho)|^p}{|\varrho-r|^{1+s\,p}}\,d\varrho\,dr+C\, {R_0}^{N-1}\,\int_{R_0}^{R_1} |u|^p\,d\varrho\\
	&\ge C\,\left(\int_{R_0}^{R_1} |u|^\frac{p}{1-s\,p}\,d\varrho\right)^{1-s\,p}.
	\end{split}
	\]
	In the last line we used \eqref{1d}. Finally, by using that $R_1<+\infty$, we get
	\[
	\begin{split}
	[u]^p_{D^{s,p}(\mathbb{R}^N)}\ge C\, \left(\int_{R_0}^{R_1} |u|^\frac{p}{1-s\,p}\,d\varrho\right)^{1-s\,p}&\ge \frac{C}{R_1^{(N-1)(1-s\,p)}}\, \left(\int_{R_0}^{R_1} |u|^\frac{p}{1-s\,p}\,\varrho^{N-1}\,d\varrho\right)^{1-s\,p}\\
	&\ge C\, \left(\int_{K} |u|^\frac{p}{1-s\,p}\,dx\right)^{1-s\,p},
	\end{split}
	\]
for some $C=C(N,s,p,R_0,R_1)>0$.
This concludes the proof in the case $s\,p<1$.
\vskip.2cm\noindent
{\it Case $s\,p=1$.} This is the same proof as before, we only need to observe that in this case, in place of \eqref{1d}, we have for every $1\le t<\infty$
\begin{equation}
\label{1dconf}
\int_{R_0}^{R_1} \int_{R_0}^{R_1} \frac{|u(r)-u(\varrho)|^p}{|\varrho-r|^{1+s\,p}}\,d\varrho\,dr+\int_{R_0}^{R_1} |u|^p\,d\varrho\ge \frac{1}{T}\,\left(\int_{R_0}^{R_1} |u|^t\,d\varrho\right)^\frac{p}{t},
	\end{equation}
for some $T=T(s,t,p,R_0,R_1)>0$. Then we can proceed as above, we leave the details to the reader.
\end{proof}

\subsection{Proof of Theorem~\ref{radial-case}}
The proof is the same as that of Theorem \ref{Theorem3}, we only need to modify {\bf Step 4} and {\bf Step 5} as follows. 
With the previous notations, as in the proof of Theorem \ref{Theorem3} we already know that $\lambda^1_n\to 0$ as $n$ goes to $\infty$. We now show that this implies that
\begin{equation}
\label{ombelico}
z_0^1=\lim_{n\to\infty} z_n^1=0.
\end{equation}
Indeed, if this was not the case, up to a subsequence, one would have $|z_n^1|\geq\tau_0$ eventually for some $\tau_0>0$.
Taking into account Proposition \ref{prop:radembedding}, observing that $p^*_s<p^\#_s$ for $N\ge 2$ and recalling that $u_n^1$ converges to $0$ almost everywhere, we have $u_n^1\to 0$ in $L^{p^*_s}(K)$ for every $K\Subset\mathbb{R}^N\setminus\{0\}$. Then, for $0<\tau<\tau_0$, we conclude
$$
\delta=\int_{B_{\lambda_n^1}(z_n^1)}|u_n^1|^{p^*_s}\,dx=\int_{B_{\lambda_n^1}(z_n^1)\cap {B_{\tau}(0)}}|u_n^1|^{p^*_s}\,dx+o_n(1)=o_n(1),
$$
since, eventually $B_{\lambda_n^1}(z_n^1)\cap {B_{\tau}(0)}=\emptyset$, thanks to the convergence of $\lambda^1_n$ to $0$.
\par
The property \eqref{ombelico} in turn implies that in {\bf Step 4} we are in the case covered by Lemma \ref{proof5}, i.e.
\[
\lim_{n\to\infty} \frac{1}{\lambda_n^1}\,\mathrm{dist}(z_n^1,\partial B)=\infty,
\]
thus we do not need Assumption {\bf (NA)} this time.
\par
Then, in order to prove \eqref{azero-r}, we need to remove the translations by $z_n^i$ from \eqref{decomposizione}. This is done by appealing to \eqref{ombelico} and continuity of $L^p$ norms with respect to translations. Indeed, by triangle inequality we have
\[
\begin{split}
\left[ u_n-v^0-\sum_{i=1}^k(\lambda_n^i)^{\frac{sp-N}{p}}v^i\left(\frac{\cdot}{\lambda_n^i}\right)\right]_{D^{s,p}(\mathbb{R}^N)}&\le \left[ u_n-v^0-\sum_{i=1}^k(\lambda_n^i)^{\frac{sp-N}{p}}v^i\left(\frac{\cdot - z_n^i}{\lambda_n^i}\right)\right]_{D^{s,p}(\mathbb{R}^N)}\\
&+\sum_{i=1}^k(\lambda_n^i)^{\frac{sp-N}{p}}\,\left[v^i\left(\frac{\cdot - z_n^i}{\lambda_n^i}\right)-v^i\left(\frac{\cdot}{\lambda_n^i}\right)\right]_{D^{s,p}(\mathbb{R}^N)}.
\end{split}
\]
By observing that the both norms converge to $0$, we get the conclusion. \qed
\vskip4pt
\noindent
\appendix

\section{A truncation Lemma}

The following result is proved in \cite[Lemma 5.3]{guida} under the stronger assumption
$u\in D^{s,p}(\mathbb{R}^N)\cap L^p(\mathbb{R}^N)$. We need to remove the last integrability assumption.
\begin{lemma}
\label{Bddlemma}
Let $\psi$ be a Lipschitz function with compact support and $u\in D^{s,p}_0(\mathbb{R}^N)$. Then $\psi\,u\in D^{s,p}_0(\mathbb{R}^N)$ and we have the estimate
\[
[\psi\,u]^p_{D^{s,p}(\mathbb{R}^N)}\le C_1\,\|\psi\|^p_{L^\infty(\mathbb{R}^N)}\,[u]^p_{D^{s,p}(\mathbb{R}^N)}+C_2\,\|\nabla \psi\|^p_{L^\infty(\mathbb{R}^N)}\,\|u\|^p_{L^{p^*_s}(\mathbb{R}^N)},
\]
for some $C_1=C_1(N,s,p)>0$ and $C_2=C_2(N,s,p,K)>0$, where $K:=\mathrm{supp}(\psi)$.
\end{lemma}
\begin{proof}
We notice that
\begin{align*}
[\psi\,u]^p_{D^{s,p}(\mathbb{R}^N)}& \leq
2^{p-1}\,\|\psi\|^p_{L^\infty(\mathbb{R}^N)}\,[u]^p_{D^{s,p}(\mathbb{R}^N)}+2^{p-1}\,\int_{\mathbb{R}^{2N}} \frac{|u(x)|^p\,|\psi(x)-\psi(y)|^p}{|x - y|^{N+s\,p}}\, dx\, dy.
\end{align*}
With a simple change of variables, the last integral can be written as 
\[
\int_{\mathbb{R}^{N}} |u(x)|^p\,\left(\int_{\mathbb{R}^N}\frac{\,|\psi(x)-\psi(x+h)|^p}{|h|^{N+s\,p}}\, dh\right)\,dx.
\]
By using H\"older inequality with exponents $p^*_s/p$ and $N/sp$, Fubini Theorem and triangle inequality, the previous integral can be estimated by
\[
\begin{split}
\int_{\mathbb{R}^{2N}} \frac{|u(x)|^p\,|\psi(x)-\psi(x+h)|^p}{|h|^{N+s\,p}}\, dx\,dh&=\int_{\mathbb{R}^N}\int_{\{|h|\le 1\}} \frac{|u(x)|^p\,|\psi(x)-\psi(x+h)|^p}{|h|^{N+s\,p}}\, dx\,dh\\
&+\int_{\mathbb{R}^N}\int_{\{|h|>1\}}\frac{|u(x)|^p\,|\psi(x)-\psi(x+h)|^p}{|h|^{N+s\,p}}\, dx\,dh\\
&\le \left(\int_{\mathbb{R}^N} |u|^{p^*_s}\,ds\right)^\frac{p}{p^*_s}\\
&\times \left(\int_{\mathbb{R}^N}\left(\int_{\{|h|\le 1\}} \frac{\,|\psi(x)-\psi(x+h)|^p}{|h|^{N+s\,p}}\, dh\right)^\frac{N}{s\,p}\,dx\right)^\frac{s\,p}{N}\\
&+2^{p-1}\,\int_{\mathbb{R}^N}\int_{\{|h|>1\}}|u(x)|^p\,\frac{|\psi(x)|^p}{|h|^{N+s\,p}}\,dh\,dx\\
&+2^{p-1}\,\int_{\mathbb{R}^N}\int_{\{|h|>1\}} |u(x)|^p\,\frac{|\psi(x+h)|^p}{|h|^{N+s\,p}}\, dx\,dh.
\end{split}
\]
For the first integral containing $\psi$, we observe that the function
\[
x\mapsto \int_{\{|h|\le 1\}} \frac{|\psi(x)-\psi(x+h)|^p}{|h|^{N+s\,p}}\,dh,
\]
is compactly supportedand bounded, indeed
\[
\int_{\{|h|\le 1\}} \frac{|\psi(x)-\psi(x+h)|^p}{|h|^{N+s\,p}}\,dh\le \|\nabla \psi\|_{L^\infty}\,\int_{\{|h|\le 1\}} |h|^{p\,(1-s)-N}\,dh=C\,\|\nabla \psi\|_{L^\infty}.
\]
For the second integral containing $\psi$, by using that $h\mapsto|h|^{N+s\,p}$ is integrable at infinity, we simply have
\[
\int_{\mathbb{R}^N}\int_{\{|h|>1\}}|u(x)|^p\,\frac{|\psi(x)|^p}{|h|^{N+s\,p}}\,dh\,dx\le C\,\|\psi\|_{L^\infty}\,\int_{K} |u|^p\,dx\le C\,|K|^{\frac{s\,p}{N}}\,\|u\|^p_{L^{p^*_s}(\mathbb{R}^N)}.
\]
For the last integral, we just observe that for every $|h|>1$, the function $\psi(\cdot+h)$ is compactly supported. We thus have
\[
\begin{split}
\int_{\mathbb{R}^N}\int_{\{|h|>1\}} |u(x)|^p\,\frac{|\psi(x+h)|^p}{|h|^{N+s\,p}}\, dx\,dh&=\int_{\{|h|>1\}} \int_{K-h}|u(x)|^p\,\frac{|\psi(x+h)|^p}{|h|^{N+s\,p}}\, dx\,dh\\
&\le \|\psi\|^p_{L^\infty}\,\int_{\{|h|>1\}} \left(\int_{K-h}|u|^p\, dx\right)\,\frac{dh}{|h|^{N+s\,p}}\\
&\le C\, \|\psi\|^p_{L^\infty}\, |K|^{\frac{s\,p}{N}}\,\|u\|^p_{L^{p^*_s}(\mathbb{R}^N)}.
\end{split}
\]
By collecting all the estimates, we conclude the proof.
\end{proof}
The following result has been curcially exploited in the proof of Theorem \ref{Theorem3}, in order to localize the rescaled sequences.
\begin{lemma}[Truncation Lemma]
\label{lm:truncation}
Let $\zeta\in C^\infty_0(B_2(0))$ be a positive function such that $\zeta\equiv 1$ on $B_1(0)$. Then
\begin{equation}
\label{tronca}
\lim_{n\to\infty}\left[v\,\zeta(\mu_n\,\cdot)- v\right]_{D^{s,p}({\mathbb R}^N)}=0,
\end{equation}
for any $v\in D^{s,p}_0({\mathbb  R}^N)\cap L^q(\mathbb{R}^N)$ with $q<p^*_s$ and  $\{\mu_n\}_{n\in{\mathbb N}}\subset{\mathbb R}^+$ such that $\mu_n\to 0$. \end{lemma}
\begin{proof}
We rewrite the term in \eqref{tronca} as $[v\,\psi_n]_{D^{s,p}({\mathbb R}^N)}$,
where $\psi_n(x):=\zeta(\mu_n\, x)-1$. We have $v\,\psi_n\in D^{s,p}_0(\mathbb{R}^N)$ thanks to Lemma \ref{Bddlemma} and
$$
\frac{|v(x)\,\psi_n(x)-v(y)\,\psi_n(y)|^p}{|x-y|^{N+s\,p}}
\leq 2^{p-1}\,|\psi_n(x)|^p\,\frac{|v(x)-v(y)|^p}{|x-y|^{N+s\,p}}
+2^{p-1}\,\frac{|\psi_n(x)-\psi_n(y)|^p\,|v(y)|^p}{|x-y|^{N+s\,p}}.
$$
Then, since $\|\psi_n\|_{L^\infty}\leq 1$ and $v\in D^{s,p}_0({\mathbb  R}^N)$, 
the Dominated Convergence Theorem yields
$$
\lim_{n\to\infty}\int_{{\mathbb R}^{2N}}	|\psi_n(x)|^p\,\frac{|v(x)-v(y)|^p}{|x-y|^{N+s\,p}}\,dx\,dy=0.
$$
For the second term, we observe that $\psi_n(x)-\psi_n(y)=\zeta(\mu_n\,x)-\zeta(\mu_n\,y)$, introduce
\[
I_{\mathbb{R}^N}(y)=\int_{\mathbb{R}^N}\frac{|\zeta(\mu_n\,x)-\zeta(\mu_n\,y)|^p}{|x-y|^{N+s\,p}}\,dx,
\] 
and decompose
\[
\begin{split}
\int_{\mathbb{R}^N} |v|^p\,I_{\mathbb{R}^N}\,dy&=
\int_{B_{1/\mu_n}} |v|^p\,I_{\mathbb{R}^N}\,dy+
\int_{B_{2/\mu_n}\setminus B_{1/\mu_n}} |v|^p\,I_{\mathbb{R}^N}\,dy+
\int_{\mathbb{R}^N\setminus B_{2/\mu_n}} |v|^p\,I_{\mathbb{R}^N}\,dy\\
&=\int_{B_{1/\mu_n}}\int_{\mathbb{R}^N\setminus B_{1/\mu_n}} |v(y)|^p\,\frac{|\zeta(\mu_n\,x)-1|^p}{|x-y|^{N+s\,p}}\,dy\,dx\\
&+\int_{B_{2/\mu_n}\setminus B_{1/\mu_n}} |v(y)|^p\,I_{\mathbb{R}^N}(y)\,dy\\
&+\int_{\mathbb{R}^N\setminus B_{2/\mu_n}}\int_{B_{2/\mu_n}} |v(y)|^p\,\frac{|\zeta(\mu_n\,x)|^p}{|x-y|^{N+s\,p}}\,dy\,dx=\mathcal{I}_1+\mathcal{I}_2+\mathcal{I}_3.
\end{split}
\]
{\bf First integral}. This is the most delicate one, here the assumption $v\in L^q(\mathbb{R}^N)$ with $q<p^*_s$ will play a major r\^ole. We have
\[
\mathcal{I}_1=\int_{B_{1/\mu_n}}\int_{B_{2/\mu_n}\setminus B_{1/\mu_n}} |v(y)|^p\,\frac{|\zeta(\mu_n\,x)-1|^p}{|x-y|^{N+s\,p}}\,dy\,dx+\int_{B_{1/\mu_n}}\int_{\mathbb{R}^N\setminus B_{2/\mu_n}} |v(y)|^p\,\frac{1}{|x-y|^{N+s\,p}}\,dy\,dx.
\]
We observe that for $y\in B_{1/\mu_n}$
\[
\begin{split}
\int_{B_{2/\mu_n}\setminus B_{1/\mu_n}} \frac{|\zeta(\mu_n\,x)-1|^p}{|x-y|^{N+s\,p}}\,dx&\le \mu_n^p\,\|\nabla \zeta\|_{L^\infty}^p\, \int_{B_{2/\mu_n}\setminus B_{1/\mu_n}} \frac{1}{|x-y|^{N+s\,p-p}}\,dx\\
&\le \mu_n^p\,\|\nabla\zeta\|_{L^\infty}^p\, \int_{B_{3/\mu_n}(y)} \frac{1}{|x-y|^{N+s\,p-p}}\,dx\\
&=C\, \mu_n^{s\,p}\,\|\nabla\zeta\|_{L^\infty}^p.
\end{split}
\]
The other term is simpler, indeed by observing that $|x-y|\ge |x|/2$ for $y\in B_{1/\mu_n}$ and $x\in\mathbb{R}^N\setminus B_{2/\mu_n}$, we have
\[
\int_{\mathbb{R}^N\setminus B_{2/\mu_n}}\frac{1}{|x-y|^{N+s\,p}}\,dx\le C\,\int_{\mathbb{R}^N\setminus B_{2/\mu_n}}\frac{1}{|x|^{N+s\,p}}\,dx=C\, \mu_n^{s\,p}.
\]
Thus we can infer
\[
\mathcal{I}_1\le C\,\mu_n^{s\,p}\,\int_{B_{1/\mu_n}} |v|^p\,dy\le C\,\mu_n^{s\,p-N+\frac{N}{q}\,p}\,\left(\int_{B_{1/\mu_n}} |v|^{q}\,\right)^\frac{p}{q},
\]
which converges to $0$, since
\[
s\,p-N+\frac{N}{q}\,p>0\ \Longleftrightarrow\ q<p^*_s.
\]
{\bf Second integral}. This is equivalent to
\[
\begin{split}
\mathcal{I}_2&=\int_{B_{2/\mu_n}\setminus B_{1/\mu_n}} \int_{B_{2/\mu_n}}|v(y)|^p\,\frac{|\zeta(\mu_n\,x)-\zeta(\mu_n\,y)|^p}{|x-y|^{N+s\,p}}\,dy\,dx\\
&+\int_{B_{2/\mu_n}\setminus B_{1/\mu_n}} \int_{\mathbb{R}^N\setminus B_{2/\mu_n}}|v(y)|^p\,\frac{|\zeta(\mu_n\,y)|^p}{|x-y|^{N+s\,p}}\,dy\,dx=\mathcal{I}_{2,1}+\mathcal{I}_{2,2}.
\end{split}
\]
For the first term, we observe that for $y\in B_{2/\mu_n}\setminus B_{1/\mu_n}$
\[
\begin{split}
\int_{B_{2/\mu_n}}\frac{|\zeta(\mu_n\,x)-\zeta(\mu_n\,y)|^p}{|x-y|^{N+s\,p}}\,dx&\le \mu_n^p\|\nabla \zeta\|^p_{L^\infty}\int_{B_{2/\mu_n}}\frac{1}{|x-y|^{N+s\,p-p}}\,dx\\
&\le \mu_n^p\,\|\nabla \zeta\|^p_{L^\infty}\,\int_{B_{4/\mu_n}(y)}\frac{1}{|x-y|^{N+s\,p-p}}\,dx\\
&\le C\,\mu_n^{s\,p}\,\|\nabla \zeta\|_{L^\infty}.
\end{split}
\]
For the other term, we observe that for $y\in B_{2/\mu_n}\setminus B_{1/\mu_n}$
\[
\begin{split}
\int_{\mathbb{R}^N\setminus B_{2/\mu_n}}\frac{|\zeta(\mu_n\,y)|^p}{|x-y|^{N+s\,p}}\,dx&\le \mu_n^p\,\|\nabla\zeta\|^p_{L^\infty}\,\int_{B_{4/\mu_n}\setminus B_{2/\mu_n}}\frac{1}{|x-y|^{N+s\,p-p}}\,dx\\
&+C\,\|\zeta\|_{L^\infty}^p\,\int_{4/\mu_n}^\infty\frac{\varrho^{N-1}}{(\varrho-|y|)^{N+s\,p}}\,d\varrho\\
&\le C\,\mu_n^{s\,p}\,\|\nabla \zeta\|^p_{L^\infty}\,+C\,\|\zeta\|^p_{L^\infty}\,\int_{4/\mu_n}^\infty\frac{\varrho^{N-1}}{(\varrho-2/\mu_n)^{N+s\,p}}\,d\varrho\\
&\le  C\,\mu_n^{s\,p}\,\left(\|\nabla \zeta\|^p_{L^\infty}+\|\zeta\|^p_{L^\infty}\right).
\end{split}
\]
In conclusion, we obtain
\[
\mathcal{I}_{2}\le C\,\mu_n^{s\,p}\,\int_{B_{2/\mu_n}\setminus B_{1/\mu_n}} |v|^{p}\,dy\le  C\,\mu_n^{s\,p}\,\mu_n^{-s\,p}\left(\int_{B_{2/\mu_n}\setminus B_{1/\mu_n}} |v|^{p^*_s}\,dy\right)^\frac{p}{p^*_s},
\]
and the latter converges to $0$, since $v\in L^{p^*_s}(\mathbb{R}^N)$.
\vskip.2cm\noindent
{\bf Third integral}. 
We proceed similarly as before for the integral in $x$, we have for every $|y|\ge 4/\mu_n$
\[
\int_{B_{2/\mu_n}}\frac{|\zeta(\mu_n\,x)|^p}{|x-y|^{N+s\,p}}\,dx\le C\,\mu_n^{-N}\,\|\zeta\|^p_{L^\infty}\,|y|^{-N-s\,p}\le C\,\|\zeta\|^p_{L^\infty}\,|y|^{-s\,p},
\]
while for $2/\mu_n\le |y|\le 4/\mu_n$ we can use the Lipschitz character of $\zeta$ and get
\[
\begin{split}
\int_{B_{2/\mu_n}}\frac{|\zeta(\mu_n\,x)|^p}{|x-y|^{N+s\,p}}\,dx&\le \mu_n^{p}\,\|\nabla\zeta\|^p_{L^\infty}\,\int_{B_{2/\mu_n}}\frac{1}{|x-y|^{N+s\,p-p}}\,dx\\
&\le \mu_n^{p}\,\|\nabla\zeta\|^p_{L^\infty}\,\int_{B_{6/\mu_n}(y)}\frac{1}{|x-y|^{N+s\,p-p}}\,dx\\
&\le C\, \mu_n^{s\,p}\,\|\nabla\zeta\|^p_{L^\infty}.
\end{split}
\] 
In conclusion we get
\[
\begin{split}
\mathcal{I}_3\le C\,\mu_n^{s\,p}\,\|\nabla \zeta\|^p_{L^\infty}\,\int_{B_{4/\mu_n}\setminus B_{2/\mu_n}} |v|^p\,dy+C\,\|\zeta\|_{L^\infty}^p\,\int_{\mathbb{R}^N\setminus B_{4/\mu_n}} \frac{|v|^p}{|y|^{s\,p}}\,dy.
\end{split}
\]
The first term tends to $0$ as before. The second one vanishes since $|v|^p\,|y|^{-s\,p}$ is integrable, thanks to Hardy inequality for $D^{s,p}(\mathbb{R}^N)$ (see \cite[Theorem 1.1]{FS}).
\end{proof}
\medskip

\section{Some regularity estimates}

We collect in this Appendix some basic regularity results for nonlocal equations needed in the paper.
\begin{proposition}
\label{prop:integrabile}
Let $1<p<\infty$ and $s\in(0,1)$ be such that $s\,p<N$. Let $E\subset\mathbb{R}^N$ be an open set with $|E|=+\infty$ and let $V\in D^{s,p}_0(E)$ be a weak solution of
\begin{equation}
\label{guazione}
\begin{cases}
(- \Delta)_p^s\, V  = \mu\, |V|^{p_s^\ast - 2}\,V, & \mbox{ in }E,\\
V=0 & \mbox{ in }\mathbb{R}^N\setminus E.
\end{cases}
\end{equation}
Then we have
\begin{equation}
\label{hey!}
V\in L^q(\mathbb{R}^N),\qquad \mbox{ for every }\ \frac{p^*_s}{p'}<q\le p^*_s.
\end{equation}
\end{proposition}
\begin{proof}
We first observe that $|V|\in D^{s,p}_0(E)$ is a positive subsolution of \eqref{guazione}, in the sense that
\begin{equation}
\label{sguazione}
\int_{\mathbb{R}^{2N}} \frac{\Big(J_p(|V(x)|-|V(y)|)\Big)\,(\varphi(x)-\varphi(y))}{|x-y|^{N+s\,p}}\,dx\,dy\le \mu\,\int_{\mathbb{R}^N} |V|^{p^*_s-1}\,\varphi\,dx,
\end{equation}
for every $\varphi \in D^{s,p}_0(E)$ positive.\footnote{This can be easily seen as in \cite{BP}.}
We can now closely follow the proof of \cite[Proposition 3.5]{BMS} for $|V|$. For $0<\alpha<1$ and 
$\varepsilon>0$, we introduce the Lipschitz increasing function
\[
\psi_\varepsilon(t)=\int_0^t \left[(\varepsilon+\tau)^\frac{\alpha-1}{p}+\frac{\alpha-1}{p}\,\tau\,(\varepsilon+\tau)^\frac{\alpha-1-p}{p}\right]^p\,d\,\tau,\qquad t\ge 0.
\]
We observe that
\begin{equation}
\label{geps}
0\le \psi_\varepsilon(t)\le \int_0^t (\varepsilon+t)^{\alpha-1}\,d\tau=\frac{1}{\alpha}\,[(\varepsilon+t)^\alpha-\varepsilon^\alpha]\le \frac{t^\alpha}{\alpha},
\end{equation}
where in the second inequality we used that $0<\alpha<1$.
We insert in \eqref{sguazione} the test function $\varphi=\psi_\varepsilon(|V|)\in D^{s,p}_0(E)$. This gives
\[
\begin{split}
\int_{\mathbb{R}^{2N}} & \frac{J_p(|V(x)|-|V(y)|)\, \big(\psi_{\varepsilon}(|V(x)|)-\psi_\varepsilon(|V(y)|\big)}{|x-y|^{N+s\,p}}\,dx\,dy\le\mu\, \int_{\mathbb{R}^N} |V|^{p^*_s-1}\,\psi_\varepsilon(|V|)\,dx.\\
\end{split}
\]
Then one needs to introduce
\[
\Psi_\varepsilon(t):=\int_0^t \psi'_\varepsilon(\tau)^\frac{1}{p}\,d\tau	=t\,(\varepsilon+t)^\frac{\alpha-1}{p},
\]
and pick a level $K_0>0$, whose precise choice will be made in a while. Observe that thanks to Chebyshev inequality, the set $\{|V|>K_0\}$ has finite measure, thus for $0<\alpha<1$ we have
\[
\int_{\{|V|>K_0\}}|V|^{p^*_s+\alpha-1}\,dx<+\infty.
\]
By using \eqref{geps} and proceeding exactly as in \cite{BMS}, we get
\[
\begin{split}
\mathcal{S}_{p,s}\,\left(\int_{\mathbb{R}^N} \Psi_\varepsilon(|V|)^{p^*_s}\,dx\right)^\frac{p}{p^*_s}&\le \frac{\mu}{\alpha}\,\int_{\{|V|>K_0\}}|V|^{p^*_s+\alpha-1}\,dx\\
&+\frac{\mu}{\alpha}\,\left(\int_{\{|V|\le K_0\}} |V|^{p^*_s}\,dx\right)^\frac{p^*_s-p}{p^*_s}\,\left(\int_{\mathbb{R}^N} \Psi_\varepsilon(|V|)^{p^*_s}\,dx\right)^\frac{p}{p^*_s},
\end{split}
\]
The level $K_0=K_0(\alpha,V)>0$ is now chosen so that
\[
\left(\int_{\{|V|\le K_0\}} |V|^{p^*_s}\,dx\right)^\frac{p^*_s-p}{p^*_s}\le \frac{\alpha}{2\,\mu}\,\mathcal{S}_{p,s},
\]
which yields
\[
\left(\int_{\mathbb{R}^N} \left(|V|\,(|V|+\varepsilon)^\frac{\alpha-1}{p}\right)^{p^*_s}\,dx\right)^\frac{p}{p^*_s}\le \frac{2\,\mu}{\alpha\,\mathcal{S}_{p,s}}\,\int_{\{|V|>K_0\}}|V|^{p^*_s+\alpha-1}\,dx,
\]
for every $0<\alpha<1$. By taking the limit as $\varepsilon$ goes to $0$, we get the desired integrability \eqref{hey!}. 
\end{proof}

\noindent
Next we state a variant of an estimate 
proved by Di Castro, Kuusi and Palatucci in \cite{DKP1}.

\begin{lemma}[Logarithmic estimate]
\label{lm:loglemma}
Let $1<p<\infty$ and $s\in(0,1)$ be such that $s\,p<N$. Let $\Omega\subset\mathbb{R}^N$ be an open bounded set, $a\in L^{N/sp}(\Omega)$ and let $u\in D^{s,p}_0(\Omega)\setminus\{0\}$ be such that
\[
\begin{cases}
(- \Delta)_p^s\, u +a\,u^{p-1}\ge \mu\,|u|^{q-2}\,u,& \mbox{ in }\Omega,\\
u  = 0,& \mbox{ in } \mathbb{R}^N \setminus \Omega,
\end{cases}
\]
for some $\mu\in\mathbb{R}$ and $p\le q\le p^*_s$. That is, for every $\varphi\in D^{s,p}_0(\Omega)$ with $\varphi\geq 0$, we have
$$
\int_{\mathbb{R}^{2N}} \frac{J_p(u(x) - u(y))\, (\varphi(x) - \varphi(y))}{|x - y|^{N+s\,p}}\, dx\, dy+\int_\Omega a\,u^{p-1}\varphi\, dx\geq \mu\,\int_\Omega |u|^{q-2}\,u\,\varphi\,dx.
$$
Let us suppose that $u\ge 0$ in $B_{2\,r}(x_0)\Subset\Omega$.
Then for every $0<\delta<1$ there holds
\begin{equation}
\begin{split}
\label{Cacciotail}
\int_{B_r(x_0)\times B_r(x_0)}& \left|\log\left(\frac{\delta+u(x)}{\delta+u(y)}\right)\right|^p\,\frac{1}{|x-y|^{N+s\,p}} \, dx\,dy\\
& \le C\, r^{N-s\,p} \left\{\delta^{1-p}\, r^{s\,p}\,\int_{\mathbb{R}^N\setminus B_{2\,r}(x_0) } \frac{u_-(y)^{p-1}}{|y-x_0|^{N+sp}}\,dy\right.\\
&\left.+ \|a_+\|_{L^{N/sp}({B_{\frac{3}{2}r}(x_0)})}+\max\{-\mu,0\}\,r^{N\,\left(1-\frac{q}{p^*_s}\right)}\, \|u\|_{L^{p^*_s}({B_{\frac{3}{2}r}(x_0)})}^{q-p}+1\right\},
\end{split}
\end{equation}
where $u_-=\max\{-u,0\}$ and $C=C(N,p,s)>0$ is a constant.
\end{lemma}
\begin{proof}
The proof is exactly the same of that of the Logarithmic Lemma for supersolutions in the case $a\equiv 0$, see \cite[Lemma 1.3]{DKP1}. We take a test function $\phi\in C^\infty_0(B_{3/2\,r}(x_0))$ such that
\[
0\le \phi\le 1,\qquad \phi\equiv 1 \mbox{ on } B_r(x_0),\qquad |\nabla \phi|\le \frac{C}{r}.
\]
Then we insert the test function $\varphi=\phi^p\,(\delta+u)^{1-p}$ in the equation. By using that
\[
\int_\Omega a\,u^{p-1}\,\frac{\phi^p}{(\delta+u)^{p-1}}\,dx\le \int_{B_{\frac{3}{2}\,r}(x_0)} a_+\,dx\le C\, r^{N-s\,p}\,\left(\int_{B_{\frac{3}{2}\,r}(x_0)} (a_+)^\frac{N}{s\,p}\,dx\right)^\frac{s\,p}{N},
\]
and\footnote{For $q=p$, we simply have
\[
\int_\Omega u^{q-1}\,\frac{\phi^p}{(\delta+u)^{p-1}}\,dx\le \int_{B_{2\,r}(x_0)} \phi^p\,dx\le c\, r^N. 
\]}
\[
\int_\Omega u^{q-1}\,\frac{\phi^p}{(\delta+u)^{p-1}}\,dx\le \int_{B_{\frac{3}{2}r}(x_0)} u^{q-p}\,\phi^p\,dx\le C\,r^{N-s\,p}\,r^{N\,\left(1-\frac{q}{p^*_s}\right)}\,\left(\int_{B_{\frac{3}{2}r}(x_0)} u^{p^*_s}\,dx\right)^\frac{q-p}{p^*_s},
\]
and proceeding exactly as in the proof of \cite[Lemma 1.3]{DKP1} to estimate the nonlocal term, we end 
up with inequality \eqref{Cacciotail}.
\end{proof}
\begin{proposition}[Minimum principle]
\label{minimumP}
Let $1<p<\infty$ and $s\in(0,1)$ be such that $s\,p<N$. Let $\Omega\subset\mathbb{R}^N$ be an open bounded connected set, $a\in L^{N/sp}(\Omega)$ and let $u\in D^{s,p}_0(\Omega)\setminus\{0\}$ be a non negative function such that
\[
\begin{cases}
(- \Delta)_p^s\, u +a\,u^{p-1}\ge \mu\,u^{q-1},& \mbox{ in }\Omega,\\
u  = 0,& \mbox{ in } \mathbb{R}^N \setminus \Omega,
\end{cases}
\]
for some $\mu\in\mathbb{R}$ and $p\le q\le p^*_s$. Then we have $u>0$ almost everywhere in $\Omega$.
\end{proposition}
\begin{proof}
The proof is the same of that for the case $a_+\equiv 0$ and $\mu\ge 0$,  which is contained in \cite[Theorem A.1]{BF}. It is sufficient to replace the logarithmic estimate there with the one of Lemma \ref{lm:loglemma} and use that $u\in L^{p^*_s}(\Omega)$. We leave the details to the reader.
\end{proof}

\end{document}